\documentclass[12pt]{article}
\usepackage{amsthm,amsfonts,amsmath,amscd,amssymb,latexsym}    
\usepackage{mathrsfs}
\usepackage{epsfig,graphicx,color}  
\usepackage[all]{xy}  
\usepackage{makeidx}
\usepackage{hyperref}

\title{The Donaldson geometric flow\\
for symplectic four-manifolds}

\author{
Robin~Krom\thanks{Partially supported by the 
Swiss National Science Foundation Grant 200021-127136}
\, \& Dietmar~A.~Salamon\\
ETH Z\"urich
}

\date{16 July 2019}

\newtheorem{PARA}{}[section] 
\newtheorem{theorem}[PARA]{Theorem} 
\newtheorem{corollary}[PARA]{Corollary} 
\newtheorem{lemma}[PARA]{Lemma} 
\newtheorem{proposition}[PARA]{Proposition} 
\newtheorem{definition}[PARA]{Definition}    
\newtheorem{remark}[PARA]{Remark}

\newcommand{\p}{\partial}
\newcommand{\one}{{{\mathchoice {\rm 1\mskip-4mu l} {\rm 1\mskip-4mu l}
{\rm 1\mskip-4.5mu l} {\rm 1\mskip-5mu l}}}}
\newcommand{\C}{{\mathbb C}}
\renewcommand{\H}{{\mathbb H}}

\newcommand{\R}{{\mathbb R}}

\newcommand{\cE}{{\mathcal E}}
\newcommand{\cG}{{\mathcal G}} 
\newcommand{\cH}{{\mathcal H}}
\newcommand{\cJ}{{\mathcal J}}
\newcommand{\cL}{{\mathcal L}} 
\newcommand{\cM}{{\mathcal M}} 
\newcommand{\cS}{{\mathcal S}}
\newcommand{\sE}{\mathscr{E}}    
\newcommand{\sF}{\mathscr{F}}    
\newcommand{\sG}{\mathscr{G}}    
\newcommand{\sH}{\mathscr{H}}    
\newcommand{\sS}{\mathscr{S}}    
\newcommand{\Om}{{\Omega}}
\newcommand{\om}{{\omega}}

\renewcommand{\phi}{{\varphi}}
\newcommand{\fhat}{{\widehat{f}}}
\newcommand{\uhat}{{\widehat{u}}}
\newcommand{\xhat}{{\widehat{x}}}
\newcommand{\Hhat}{{\widehat{H}}}
\newcommand{\Khat}{{\widehat{K}}}
\newcommand{\rhohat}{{\widehat{\rho}}}
\newcommand{\lahat}{{\widehat{\lambda}}}
\newcommand{\Thetahat}{{\widehat{\Theta}}}
\newcommand{\tom}{{\widetilde{\om}}}
\newcommand{\tJ}{{\widetilde{J}}}
\newcommand{\tg}{{\widetilde{g}}}
\newcommand{\im}{{\mathrm{im}}} 
\newcommand{\id}{{\mathrm{id}}} 
\newcommand{\Vol}{{\mathrm{Vol}}}  
\newcommand{\ndg}{{\mathrm{ndg}}} 
\newcommand{\grad}{{\mathrm{grad}}}  
\newcommand{\Lie}{{\mathrm{Lie}}} 
\newcommand{\Diff}{{\mathrm{Diff}}} 
\newcommand{\Vect}{{\mathrm{Vect}}}  
\newcommand{\Symp}{{\mathrm{Symp}}}   
\newcommand{\Ham}{{\mathrm{Ham}}}  
\newcommand{\Hom}{{\mathrm{Hom}}}   
\newcommand{\FS}{{\mathrm{FS}}}
\renewcommand{\i}{{\mathbf{i}}}
\renewcommand{\j}{{\mathbf{j}}}
\renewcommand{\k}{{\mathbf{k}}}
\newcommand{\dvol}{{\rm dvol}}
\newcommand{\G}{{\mathrm{G}}} 
\newcommand{\GL}{{\mathrm{GL}}} 
\newcommand{\Sp}{{\mathrm{Sp}}} 
\newcommand{\fg}{{\mathfrak{g}}} 
\newcommand{\Cinf}{C^{\infty}}
\newcommand{\CP}{{\C\mathrm{P}}}
\newcommand{\inner}[2]{\langle #1, #2\rangle}
\newcommand{\INNER}[2]{\left\langle #1, #2\right\rangle}
\def\NABLA#1{{\mathop{\nabla\kern-.5ex\lower1ex\hbox{$#1$}}}}
\def\Nabla#1{\nabla\kern-.5ex{}_{#1}}
\def\abs#1{\mathopen|#1\mathclose|}
\def\Abs#1{\left|#1\right|}

\def\Norm#1{\left\|#1\right\|}
\begin{document}
\maketitle

\begin{abstract}
This is an exposition of the Donaldson geometric flow
on the space of symplectic forms on a closed smooth 
four-manifold, representing a fixed cohomology class. 
The original work appeared in~\cite{DON1}.
\end{abstract}


\section{Introduction}\label{sec:INTRO}

For any closed symplectic four-manifold $(M,\om)$ it is an open 
question whether the space of symplectic forms on $M$
representing the same cohomology class as $\om$ is connected.
By Moser isotopy a positive answer to this question is equivalent 
to the assertion that every symplectic form in the cohomology class
of $\om$ is diffeomorphic to $\om$ via a diffeomorphism that 
is isotopic to the identity.  In the case of the projective plane
it follows from theorems of Gromov and Taubes that a positive 
answer is equivalent to the assertion that a diffeomorphism 
is isotopic to the identity if and only if it induces the identity on homology.
In the case of the four-torus a positive answer is a longstanding 
conjecture in symplectic topology.   This is part of the circle of 
questions around the uniqueness problem in symplectic
topology as discussed in~\cite{SAL}.  A remarkable 
geometric flow approach to the uniqueness problem in dimension 
four was explained by Donaldson in a lecture in Oxford in the 
spring of 1997 (attended by the second author) and written up 
in~\cite{DON1}. The purpose of this expository paper is to explain
some of the details.

\bigbreak

The starting point of Donaldson's approach is 
the obervation that the space of diffeomorphism of a hyperK\"ahler 
surface $M$ can be viewed as an infinite-dimensional hyperK\"ahler
manifold, that the group of symplectomorphisms associated 
to a preferred symplectic structure $\om$ acts on the right by 
hyperK\"ahler isometries, and that this group action is generated 
by a hyperK\"ahler moment map.  
In analogy to the finite-dimensional setting one can then study 
the gradient flow of the square of the hyperK\"ahler moment map.  
Pushing $\om$ forward under the diffeomorphisms of $M$
one obtains a geometric flow on the space $\sS_a$ of symplectic forms
in the cohomology class $a:=[\om]$.   It turns out that
this geometric flow is well defined for each symplectic four-manifold
$(M,\om)$ equipped with a Riemannian metric $g$ that
is compatible with $\om$.  It is the gradient flow of the 
energy functional
\begin{equation}\label{eq:energy}
\sE(\rho) := \int_M
\frac{2\abs{\rho^+}^2}{\abs{\rho^+}^2-\abs{\rho^-}^2}\,\dvol,
\qquad \rho\in\sS_a,
\end{equation}
with respect to a suitable metric on $\sS_a$.
To describe this metric, we recall the well known observation
(also used in~\cite{DON2})  that for every positive 
rank-$3$ subbundle ${\Lambda^+\subset\Lambda^2T^*M}$
and every positive volume form $\dvol\in\Om^4(M)$ there is 
a unique Riemannian metric on~$M$ with volume form $\dvol$ 
such that~$\Lambda^+$ is the bundle of self-dual $2$-forms
(Theorem~\ref{thm:DONfour}). 
Second, every $\rho\in\sS_a$ 
determines an involution $R^\rho:\Om^2(M)\to\Om^2(M)$ which 
sends $\rho$ to $-\rho$ and acts as the identity on the orthogonal 
complement of $\rho$ with respect to the exterior product; it is given 
by $R^\rho\tau:=\tau-2\frac{\tau\wedge\rho}{\rho\wedge\rho}\rho$
and preserves the pairing. Thus every $\rho\in\sS_a$ determines 
a unique Riemannian metric $g^\rho$ on $M$ with the same volume 
form as $g$ such that $\tau$ is self-dual with respect to $g$
if and only if $R^\rho\tau$ is self-dual with respect to $g^\rho$
(Theorem~\ref{thm:FOUR}).  For $\rho\in\sS_a$ denote by 
$*^\rho:\Om^k(M)\to\Om^{4-k}(M)$ the Hodge $*$-operator of~$g^\rho$.
Then the {\bf Donaldson metric} on the infinite-dimensional 
manifold $\sS_a$ is given by 
\begin{equation}\label{eq:metric}
\Norm{\rhohat}_\rho^2 := \int_M\lambda\wedge *^\rho\lambda,\qquad
d\lambda=\rhohat,\qquad *^\rho\lambda\mbox{ is exact},
\end{equation}
for $\rhohat\in T_\rho\sS_a=\im(d:\Om^1(M)\to\Om^2(M))$
(see Definition~\ref{def:DONmet}). 
Now the differential of the energy functional $\sE:\sS_a\to\R$ 
at a point $\rho\in\sS_a$ is the linear map 
$\rhohat\mapsto\int_M\Theta^\rho\wedge\rhohat$
where the $2$-form $\Theta^\rho\in\Om^2(M)$ is given by
\begin{equation}\label{eq:Thetarho}
\Theta^\rho := *\frac{\rho}{u}-\frac12\Abs{\frac{\rho}{u}}^2\rho,\qquad
u := \frac{\rho\wedge\rho}{2\dvol}.
\end{equation}
This is the pointwise orthogonal projection of the $2$-form $u^{-1}*\rho$
onto the orthogonal complement of $\rho$ with respect to the exterior product.

\bigbreak

The negative gradient flow of the energy functional 
$\sE:\sS_a\to\R$ in~\eqref{eq:energy} with respect 
to the Donaldson metric~\eqref{eq:metric} has the form 
\begin{equation}\label{eq:dongeoflow}
\p_t\rho = d*^\rho d\Theta^\rho,
\end{equation}
where $\Theta^\rho\in\Om^2(M)$ is given by~\eqref{eq:Thetarho}
(Proposition~\ref{prop:DGF}).  This is the {\bf Donaldson geometric flow}.
The purpose of the present paper is to explain some of the 
geometric properties of this flow, and to give an exposition 
of the necessary background material. 
This includes a discussion of the Riemannian metrics~$g^\rho$ 
which is relegated to Appendix~\ref{app:FOUR}. 
The Donaldson geometric flow in the original
hyperK\"ahler moment map setting is explained 
in Section~\ref{sec:MOMENT}, for general symplectic 
four-manifolds it is discussed in Section~\ref{sec:GENERAL}, 
and  the Hessian of the energy functional is examined 
in Section~\ref{sec:HESSIAN}. 

The motivation for this study is the dream that the solutions 
of~\eqref{eq:dongeoflow} can be used to settle 
the uniqueness problem for symplectic structures in dimension
four in some favourable cases such as hyperK\"ahler surfaces 
or the complex projective plane.  This is backed up by the 
observations that the symplectic form $\om$ is the unique absolute 
minimum of $\sE$ (Corollary~\ref{cor:DGF}) and the Hessian of 
$\sE$ at $\om$ is positive definite (Corollary~\ref{cor:MIN}). 
For $M=\CP^2$ we prove that the Fubini--Study form is the 
only critical point (Proposition~\ref{prop:CP2}).
The present exposition also includes a proof of Donaldson's 
observation that higher critical points cannot be strictly stable 
in the hyperK\"ahler setting (Theorem~\ref{thm:DON}).  
Local existence and uniqueness and regularity for the solutions 
of~\eqref{eq:dongeoflow} are estabished in the followup 
paper~\cite{KROM} for which the present paper provides 
the necessary background. Key problems for future research 
include long-time existence and to show that the solutions 
cannot {\it escape to infinity}.

{\bf Sign Conventions.}\label{notation}
Let $(M,\om)$ be a symplectic manifold and let $\G$
be a Lie group with Lie algebra $\fg:=\Lie(\G)$ that acts 
covariantly on $M$ by symplectomorphisms. Denote the 
infinitesimal action by ${\fg\to\Vect(M):\xi\mapsto X_\xi}$.
We use the sign convention ${[X,Y]:=\Nabla{Y}X-\Nabla{X}Y}$
for the Lie bracket of vector fields so the 
infinitesimal action is a Lie algebra homomorphism. We use the sign 
convention ${\iota(X_H)\om=dH}$ for Hamiltonian vector 
fields so the map ${\Cinf(M)\to\Vect(M):H\mapsto X_H}$
is a Lie algebra homomorphism with respect to the Poisson
bracket ${\{F,G\}:=\om(X_F,X_G)}$. The group action is called
{\it Hamiltonian} if there is a $\G$-equivariant
{\it moment map} $\mu:M\to\fg^*$ such that~$X_\xi$ is the 
Hamiltonian vector field of $H_\xi:=\inner{\mu}{\xi}$ for $\xi\in\fg$. 
If $\fg$ is equipped with an invariant inner product 
it is convenient to write $\mu:M\to\fg$.

{\bf Acknowledgement.}
Thanks to Simon Donaldson for many enlightening discussions.
 
\newpage


\section{The Moment Map Picture}\label{sec:MOMENT}

Throughout this section $M$ denotes a closed hyperK\"ahler surface
with symplectic forms $\om_1,\om_2,\om_3$ and complex structures
$J_1,J_2,J_3$.  Thus each $J_i$ is compatible with $\om_i$,
the resulting Riemannian metric 
$$
\inner{\cdot}{\cdot} := \om_i(\cdot,J_i\cdot)
$$ 
is independent of~$i$, and the complex structures satisfy the quaternion 
relations $J_iJ_j=-J_jJ_i=J_k$ for every cyclic permutation $i,j,k$ of $1,2,3$.
Let $(S,\sigma)$ be a symplectic four-manifold that is diffeomorphic
to $M$ and define
\begin{equation}\label{eq:F}
\sF := \left\{f:S\to M\,\Big|\,\begin{array}{l}
f\mbox{ is a diffeomorphism and} \\
\mbox{the $2$-form }f^*\om_1-\sigma\mbox{ is exact}
\end{array}\right\}.
\end{equation}
This space need not be connected.  Assume it is nonempty. 
(Whether this implies that $(S,\sigma)$ is symplectomorphic 
to $(M,\om_1)$ is an open question.)  Then the space $\sF$ is a 
$C^1$ open set in the space of all smooth maps from $S$ to $M$ 
and can be viewed formally as an infinite-dimensional hyperK\"ahler 
manifold. Its tangent space at $f\in\sF$ is the space of vector fields 
along $f$ and will be denoted by
$
T_f\sF = \Om^0(S,f^*TM).
$
The three complex structures are given by 
$T_f\sF\to T_f\sF:\fhat\mapsto J_i\fhat$ and the three
symplectic forms are given by 
\begin{equation}\label{eq:Omi}
\Om_i(\fhat_1,\fhat_2) 
:= \int_S\om_i(\fhat_1,\fhat_2)\,\dvol_\sigma,\qquad
\dvol_\sigma := \frac{\sigma\wedge\sigma}{2}
\end{equation}
for $\fhat_1,\fhat_2\in T_f\sF$.  The group
\begin{equation}\label{eq:G}
\sG := \Symp(S,\sigma) 
:= \left\{\phi\in\Diff(S)\,\big|\,\phi^*\sigma=\sigma\right\}
\end{equation}
of symplectomorphism of $(S,\sigma)$ acts contravariantly 
on $\sF$ by composition on the right.  This group action 
preserves the hyperK\"ahler structure of $\sF$.
The quotient space $\sF/\sG$ is homeomorphic
to the space $\sS$ of all symplectic forms on $M$ that are 
cohomologous to $\om_1$ and diffeomorphic to~$\sigma$ via the 
homeomorphism $\sF/\sG\to\sS:[f]\mapsto(f^{-1})^*\sigma$. 
The action of the subgroup
\begin{equation}\label{eq:G0}
\sG_0 := \Ham(S,\sigma) 
\end{equation}
of Hamiltonian symplectomorphisms is Hamiltonian 
for all three symplectic forms on $\sF$.
This is the content of the next proposition.
We identify the Lie algebra of $\sG_0$ with the space 
of smooth real valued functions on $S$ with mean value
zero and its dual space with the quotient $\Om^0(S)/\R$ 
via the $L^2$ inner product associated to the volume
form $\dvol_\sigma$. 

\begin{proposition}[{\bf Moment Map}]\label{prop:moment}
The map 
\begin{equation}\label{eq:mui}
\mu_i:\sF\to\Om^0(S),\qquad
\mu_i(f):=\frac{f^*\om_i\wedge\sigma}{\dvol_\sigma},
\end{equation}
is a moment map for the covariant action 
$$
\sG_0\times\sF\to\sF:(\phi,f)\mapsto f\circ\phi^{-1}
$$
with respect to the symplectic form $\Om_i$.
\end{proposition}

\begin{proof}
The infinitesimal covariant action of a smooth function 
$H:S\to\R$ with mean value zero on $\sF$ is given by 
the vector field on $\sF$ which assigns to each
$f\in\sF$ the vector field $-df\circ X_H\in T_f\sF$
along $f$.  Here $X_H\in\Vect(S)$ is the Hamiltonian vector 
field on $S$ associated to $H$ and is determined by the equation
$\iota(X_H)\sigma=dH$.  The minus sign appears because 
composition on the right defines a contravariant action of 
$\sG_0$ and the covariant action is
given by composition with $\phi^{-1}$ on the right.
By Cartan's formula, the differential of the map 
$\mu_i:\sF\to\Om^0(S)$ in~\eqref{eq:mui} at $f$ 
in the direction $\fhat\in T_f\sF$ is given by
\begin{equation}\label{eq:dmui}
d\mu_i(f)\fhat = \frac{d\alpha_i\wedge\sigma}{\dvol_\sigma},\qquad
\alpha_i := \om_i(\fhat,df\cdot)\in\Om^1(S).
\end{equation}
Now contract the vector field $f\mapsto -df\circ X_H$ 
on $\sF$ with the symplectic form~$\Om_i$ to obtain
\begin{eqnarray*}
\Om_i(-df\circ X_H,\fhat)
&=&
\int_S\om_i(\fhat,df\circ X_H)\,\dvol_\sigma \\
&=&
\int_S(\iota(X_H)\alpha_i)\dvol_\sigma \\
&=&
\int_S\alpha_i\wedge\iota(X_H)\dvol_\sigma \\
&=&
\int_S\alpha_i\wedge dH\wedge\sigma \\
&=&
\int_S Hd\alpha_i\wedge\sigma \\
&=&
\int_SHd\mu_i(f)\fhat\,\dvol_\sigma.
\end{eqnarray*}
The last term is the differential of the function
$\sF\to\R:f\mapsto\inner{\mu_i(f)}{H}_{L^2}$ 
at $f$ in the direction $\fhat$.
This proves Proposition~\ref{prop:moment}.
\end{proof}

The norm squared of the moment map in the hyperK\"ahler setting
is the function 
$
\sE:=\tfrac12(\Norm{\mu_1}^2+\Norm{\mu_2}^2+\Norm{\mu_3}^2),
$
where the norm on the (dual of the) Lie algebra is associated 
to an invariant inner product.  In the case at hand
the invariant inner product is the $L^2$ inner product
on $\Om^0(S)$ and the norm squared of the moment map 
is the energy functional $\sE:\sF\to\R$ given by
\begin{equation}\label{eq:energyf}
\sE(f) := \frac12\int_S\sum_{i=1}^3\abs{H_i}^2\,\dvol_\sigma,\qquad
H_i := \frac{f^*\om_i\wedge\sigma}{\dvol_\sigma}.
\end{equation}
We next examine the negative gradient flow lines of $\sE$ 
with respect to the hyperK\"ahler metric on $\sF$, given by
$$
\INNER{\fhat_1}{\fhat_2}_{L^2} 
:= \int_S\inner{\fhat_1}{\fhat_2}\,\dvol_\sigma
\qquad\mbox{for }\fhat_1,\fhat_2\in T_f\sF.
$$ 

\begin{proposition}[{\bf Gradient Flow}]\label{prop:gradientf}
An isotopy $\R\to\sF:t\mapsto f_t$ is a negative 
$L^2$ gradient flow line of the energy functional $\sE:\sF\to\R$
in~\eqref{eq:energyf} if and only if it satisfies the partial 
differential equation
\begin{equation}\label{eq:gradientf}
\p_tf_t = \sum_{i=1}^3J_idf_t\circ X_{H_{it}},\qquad 
H_{it} := \frac{f_t^*\om_i\wedge\sigma}{\dvol_\sigma},\qquad
\iota(X_{H_{it}})\sigma = dH_{it}.
\end{equation}
\end{proposition}

\begin{proof}
The differential of the energy functional $\sE:\sF\to\R$
at $f\in\sF$ in the direction $\fhat\in T_f\sF$ is given by
\begin{eqnarray*}
\delta\sE(f)\fhat 
&=& 
\sum_{i=1}^3\INNER{d\mu_i(f)\fhat}{\mu_i(f)}_{L^2} \\
&=& 
\sum_{i=1}^3\Om_i(-df\circ X_{H_i},\fhat) \\
&=& 
-\sum_{i=1}^3\INNER{J_idf\circ X_{H_i}}{\fhat}_{L^2}.
\end{eqnarray*}
Here $H_i:=\mu_i(f)\in\Om^0(S)$ is as in~\eqref{eq:energyf},
the second equation follows from Proposition~\ref{prop:moment},
and the third equation follows from the fact that
$\om_i=\inner{J_i\cdot}{\cdot}$. Hence the $L^2$ gradient 
of $\sE$ is given by
\begin{equation}\label{eq:gradEf}
\grad\sE(f) = -\sum_{i=1}^3J_idf\circ X_{H_i}
\end{equation}
and this proves Proposition~\ref{prop:gradientf}.
\end{proof}

The energy functional~\eqref{eq:energyf} and the $L^2$
metric on $\sF$ are invariant under the action of the 
full group $\sG$ of all symplectomorphisms of $(S,\sigma)$
and so is the negative gradient flow~\eqref{eq:gradientf}. 
To eliminate the action of the infinite-dimensional 
symplectomorphism group it is convenient to replace 
the solutions $t\mapsto f_t$ of equation~\eqref{eq:gradientf}
by paths of symplectic forms $t\mapsto\rho_t$ on $M$
obtained by pushing forward the symplectic form $\sigma$ on $S$
by the diffeomorphisms $f_t:S\to M$. 

\begin{proposition}[{\bf Pushforward Gradient Flow}]\label{prop:frho}
Let $\R\to\sF:t\mapsto f_t$ be a solution of~\eqref{eq:gradientf}
and define the symplectic form $\rho_t\in\Om^2(M)$ by 
$$
\rho_t := (f_t^{-1})^*\sigma
$$
for $t\in\R$. Then $\rho_t$ is cohomologous to $\om_1$ for all $t$
and the path $t\mapsto\rho_t$ satisfies the partial differential
equation
\begin{equation}\label{eq:gradientrho}
\p_t\rho_t = -\sum_{i=1}^3 d(dK_i^{\rho_t}\circ J_i^{\rho_t}),\quad 
K_i^{\rho_t} := \frac{\om_i\wedge\rho_t}{\dvol_{\rho_t}},\quad
\rho_t(J_i^{\rho_t}\cdot,\cdot) := \rho_t(\cdot,J_i\cdot).
\end{equation}
\end{proposition}

\begin{proof}
Differentiate the equation $f_t^*\rho_t=\sigma$ 
using Cartan's formula to obtain
\begin{equation}\label{eq:beta}
0 = f_t^*\p_t\rho_t + d\beta_t,\qquad 
\beta_t := \rho_t(\p_tf_t,df_t\cdot)\in\Om^1(S).
\end{equation}
Since $f_t$ satisfies~\eqref{eq:gradientf}
it follows that
\begin{eqnarray*}
\beta_t 
&=& 
\sum_{i=1}^3\rho_t(J_idf_t\circ X_{H_{it}},df_t\cdot) \\
&=& 
\sum_{i=1}^3\rho_t(df_t\circ X_{H_{it}},J_i^{\rho_t}df_t\cdot) \\
&=& 
\sum_{i=1}^3\sigma(X_{H_{it}},f_t^*J_i^{\rho_t}\cdot) \\
&=& 
\sum_{i=1}^3dH_{it}\circ f_t^*J_i^{\rho_t} \\
&=&
\sum_{i=1}^3f_t^*\bigl(dK_i^{\rho_t}\circ J_i^{\rho_t}\bigr).
\end{eqnarray*}
Here the last equation follows from the fact that
$H_{it}=K_i^{\rho_t}\circ f_t=f_t^*K_i^{\rho_t}$. Now insert 
the formula $\beta_t=\sum_if_t^*(dK_i^{\rho_t}\circ J_i^{\rho_t})$ 
into equation~\eqref{eq:beta} to obtain~\eqref{eq:gradientrho}.  
This proves Proposition~\ref{prop:frho}.
\end{proof}

\bigbreak

Equation~\eqref{eq:gradientrho} is the 
{\bf Donaldson Geometric Flow} in the hyperK\"ahler
setting.  It can be interpreted as the gradient flow 
of the pushforward energy functional on the space $\sS_a$
of all symplectic forms on $M$ representing the cohomology
class $a:=[\om_1]$ with respect to a suitable Riemannian
metric. (See Definition~\ref{def:DONmet} below.)  The energy functional 
and the Riemannian metric on $\sS_a$ are independent of the choice 
of the symplectic four-manifold $(S,\sigma)$.

\begin{proposition}[{\bf Pushforward Energy}]\label{prop:E}
Let $f\in\sF$ and define 
$$
\rho:=(f^{-1})^*\sigma\in\Om^2(M). 
$$
Then
\begin{equation}\label{eq:E}
\sE(\rho) := \sE(f) =  \int_M\frac{2\abs{\rho^+}^2}
{\abs{\rho^+}^2-\abs{\rho^-}^2}\,\dvol, 
\end{equation}
where $\rho^\pm:=\tfrac12(\rho\pm*\rho)$ are the 
self-dual and anti-self-dual parts of $\rho$ 
and $\dvol:=\tfrac12\om_i\wedge\om_i$ 
is the volume form of the hyperK\"ahler metric.
\end{proposition}

\begin{proof}
Define 
\begin{equation}\label{eq:u}
u :=\frac{\dvol_\rho}{\dvol},\qquad 
\dvol_\rho := \frac{\rho\wedge\rho}{2},
\end{equation}
and
$$
K_i^\rho:=\frac{\om_i\wedge\rho}{\dvol_\rho},\qquad
H_i:=f^*K_i^\rho=\frac{f^*\om_i\wedge\sigma}{\dvol_\sigma}
$$ 
as in~\eqref{eq:gradientrho} and~\eqref{eq:energyf}. Then 
\begin{equation}\label{eq:urho}
\rho^+=\frac{u}{2}\sum_iK^\rho_i\om_i,\qquad
2\abs{\rho^+}^2=u^2\sum_i\abs{K^\rho_i}^2,\qquad
\abs{\rho^+}^2-\abs{\rho^-}^2=2u.
\end{equation}
Divide and integrate to obtain 
\begin{eqnarray*}
\int_M\frac{2\abs{\rho^+}^2}{\abs{\rho^+}^2-\abs{\rho^-}^2}\dvol
&=& 
\int_M\frac{u}{2}\sum_i\abs{K^\rho_i}^2\,\dvol 
=  
\int_M\frac{1}{2}\sum_i\abs{K^\rho_i}^2\,\dvol_\rho \\
&=&
\frac12\sum_{i=1}^3\int_S\abs{H_i}^2\,\dvol_\sigma 
=
\sE(f).
\end{eqnarray*}
This proves Proposition~\ref{prop:E}.
\end{proof}

The energy functional $\sE:\sS_a\to\R$ in~\eqref{eq:E} 
is well defined for symplectic forms on any closed oriented 
Riemannian four-manifold $M$. Moreover, the space 
$\sS_a$ carries a natural Riemannian metric which in the hyperK\"ahler
case agrees with the pushforward of the $L^2$ metric on $\sF$ .
Thus the Donaldson geometric flow extends to the general 
setting as explained in the next section. 


\section{General Symplectic Four-Manifolds}\label{sec:GENERAL}

Let $M$ be a closed oriented Riemannian four-manifold.
Denote by $g$ the Riemannian metric on $M$, 
denote by ${\dvol\in\Om^4(M)}$ the volume form of $g$,
and let $*:\Om^k(M)\to\Om^{4-k}(M)$ be the 
Hodge $*$-operator associated to the metric and orientation.
Fix a cohomology class $a\in H^2(M;\R)$ 
such that $a^2>0$ and consider the space
$$
\sS_a:=\left\{\rho\in\Om^2(M)\,\big|\,
d\rho=0,\,\rho\wedge\rho>0,\,[\rho]=a\right\}
$$
of symplectic forms on $M$ representing the class $a$.  
This is an infinite-dimensional manifold and the 
tangent space of~$\sS_a$ at any element~${\rho\in\sS_a}$ 
is the space of exact $2$-forms on $M$.  
The next proposition is of preparatory nature. 
It summarizes the properties of a family of 
Riemannian metrics $g^\rho$ on~$M$, one for each 
nondegenerate $2$-form $\rho$ (and for each fixed 
background metric $g$). These Riemannian metrics will play a 
central role in our study of the Donaldson geometric flow.

\begin{proposition}[{\bf Symplectic Forms and Riemannian Metrics}]
\label{prop:grho}\ 

\noindent
Fix a nondegenerate $2$-form $\rho\in\Om^2(M)$
such that $\rho\wedge\rho>0$ and define the 
function $u:M\to(0,\infty)$ by~\eqref{eq:u}.
Then there exists a unique Riemannian metric 
$g^\rho$ on $M$ that satisfies the following conditions.

\smallskip\noindent{\bf (i)}
The volume form of $g^\rho$ agrees with the volume form of $g$.

\smallskip\noindent{\bf (ii)}
The Hodge $*$-operator $*^\rho:\Om^1(M)\to\Om^3(M)$
associated to $g^\rho$ is given
by 
\begin{equation}\label{eq:*rhola}
*^\rho\lambda = \frac{\rho\wedge*(\rho\wedge\lambda)}{u}
\end{equation}
for $\lambda\in\Om^1(M)$ and by
$*^\rho\iota(X)\rho = -\rho\wedge g(X,\cdot)$
for $X\in\Vect(M)$.

\smallskip\noindent{\bf (iii)}
The Hodge $*$-operator $*^\rho:\Om^2(M)\to\Om^2(M)$
associated to $g^\rho$ is given by 
\begin{equation}\label{eq:*rhom}
*^\rho\om = R^\rho*R^\rho\om,\qquad 
R^\rho\om := \om-\frac{\om\wedge\rho}{\dvol_\rho}\rho,
\end{equation}
for $\om\in\Om^2(M)$. The linear map $R^\rho:\Om^2(M)\to\Om^2(M)$ 
is an involution that preserves the exterior product, 
acts as the identity on the orthogonal complement
of $\rho$ with respect to the exterior product, 
and sends $\rho$ to $-\rho$. 

\smallskip\noindent{\bf (iv)}
Let $\om\in\Om^2(M)$ be a nondegenerate $2$-form and let
$J:TM\to TM$ be an almost complex structure such that
$g = \om(\cdot,J\cdot)$.  Define the almost complex 
structure $J^\rho$ by $\rho(J^\rho\cdot,\cdot):=\rho(\cdot,J\cdot)$
and define the $2$-form $\om^\rho\in\Om^2(M)$ by $\om^\rho:=R^\rho\om$.
Then $g^\rho = \om^\rho(\cdot,J^\rho\cdot)$ and so
$\om^\rho$ is self-dual with respect to $g^\rho$.
\end{proposition}

\begin{proof}
See Theorem~\ref{thm:FOUR}.
\end{proof}

\bigbreak

\begin{definition}\label{def:DONmet}
Each nondegenerate $2$-form $\rho\in\Om^2(M)$ with $\rho^2>0$
determines an inner product $\inner{\cdot}{\cdot}_\rho$ on the space 
of exact $2$-forms defined by
\begin{equation}\label{eq:innerho}
\INNER{\rhohat_1}{\rhohat_2}_\rho
:= \int_M\lambda_1\wedge *^\rho\lambda_2,\qquad
d\lambda_i=\rhohat_i,\qquad *^\rho\lambda_i\in\im\,d.
\end{equation}
These inner products determine a Riemannian metric on the
infinite-di\-men\-si\-onal manifold $\sS_a$ called 
the {\bf Donaldson metric}.
\end{definition}

The Donaldson geometric flow on a general symplectic four-manifold
is the negative gradient flow of the energy functional $\sE:\sS_a\to\R$
in~\eqref{eq:E} with respect to the Donaldson metric 
in Definition~\ref{def:DONmet}. A central geometric 
ingredient in this flow is the following map 
$\Theta:\Om^2_{\ndg}(M)\to\Om^2(M)$.
Its domain is the space 
$\Om^2_{\ndg}(M) := \{\rho\in\Om^2(M)\,|\,\rho\wedge\rho>0\}$
of nondegenerate $2$-forms compatible with the
orientation and the map is given by
\begin{equation}\label{eq:Theta}
\Theta(\rho) := \Theta^\rho := *\frac{\rho}{u}
- \frac{1}{2}\Abs{\frac{\rho}{u}}^2\rho,\qquad
u := \frac{\dvol_\rho}{\dvol}.
\end{equation}

\begin{proposition}[{\bf The Map $\Theta$}]\label{prop:Theta}
Let $\rho\in\Om^2_\ndg(M)$ and define $u\in\Om^0(M)$
and $\Theta^\rho\in\Om^2(M)$ by~\eqref{eq:Theta}.
Then the following holds.

\smallskip\noindent{\bf (i)}
$\Theta^\rho$ is the pointwise orthogonal projection of the 
$2$-form $u^{-1}*\rho$ onto the orthogonal complement 
of $\rho$ with respect to the wedge product. In particular
\begin{equation}\label{eq:Theta0}
\Theta^\rho\wedge\rho = 0.
\end{equation}

\smallskip\noindent{\bf (ii)}
The $2$-form $\Theta^\rho$ can be written as
\begin{equation}\label{eq:Theta1}
\Theta^\rho 
= \frac{2\rho^+}{u} -\Abs{\frac{\rho^+}{u}}^2\rho
= -\frac{\abs{\rho^-}^2\rho^++\abs{\rho^+}^2\rho^-}{u^2}.
\end{equation}

\smallskip\noindent{\bf (iii)}
The square of $\Theta^\rho$ is given by
\begin{equation}\label{eq:Theta2}
\Theta^\rho\wedge\Theta^\rho 
= -\frac{2\abs{\rho^+}^2\abs{\rho^-}^2}{u^3}\dvol.
\end{equation}
Thus $\Theta^\rho\wedge\Theta^\rho\le0$ 
with equality if and only if $\rho$ is self-dual.

\smallskip\noindent{\bf (iv)}
Let $\rho_t : \R \to \Om^2_{\ndg}(M)$ be a smooth path 
with $\rho_0 = \rho$ and ${\p_t\rho_t|_{t=0} = \rhohat}$. 
Then
\begin{equation}\label{eq:Theta3}
\Thetahat 
:= \left.\frac{d}{dt}\right|_{t=0} \Theta^{\rho_t} 
= \frac{\rhohat + *^{\rho}\rhohat}{u} 
- \Abs{\frac{\rho^+}{u}}^2 \rhohat.
\end{equation}

\smallskip\noindent{\bf (v)}
Assume the hyperK\"ahler case.  Then 
\begin{equation}\label{eq:Theta4}
\Theta^\rho 
= \sum_{i=1}^3\left(K^\rho_i\om_i - \frac12(K^\rho_i)^2\rho\right),\qquad
K^\rho_i := \frac{\om_i\wedge\rho}{\dvol_\rho},
\end{equation}
and
\begin{equation}\label{eq:Theta5}
d\Theta^\rho = *^\rho\sum_{i=1}^3dK^\rho_i\circ J_i^\rho,\qquad
\rho(J^\rho_i\cdot,\cdot) := \rho(\cdot,J_i\cdot).  
\end{equation}
\end{proposition}

\begin{proof}
It follows directly from the definition of $\Theta^\rho$ 
in~\eqref{eq:Theta} that ${\Theta^\rho\wedge\rho = 0}$
and this proves part~(i). 

To prove part~(ii), denote by $\Pi:\Om^2(M)\to\Om^2(M)$ 
the pointwise orthogonal projection onto the orthogonal 
complement of $\rho$ with respect to the wedge product. Thus 
$$
\Pi(\tau) = \tau - \frac{\tau\wedge\rho}{\rho\wedge\rho}\rho
\qquad\mbox{for }\tau\in\Om^2(M).
$$
Since $*\rho=\rho^+-\rho^-=2\rho^+-\rho$ it follows
from part~(i) that
$$
\Theta^\rho=\Pi\left(*\frac{\rho}{u}\right) 
= \Pi\left(\frac{2\rho^+}{u}\right).
$$
This proves the first equation in~\eqref{eq:Theta1}.
The second equation in~\eqref{eq:Theta1} follows by direct calculation,
using the identity $2u=\abs{\rho^+}^2-\abs{\rho^-}^2$.
This proves part~(ii). 

To prove part~(iii), use the last term in 
equation~\eqref{eq:Theta1} to obtain
$$
\Theta^\rho\wedge\Theta^\rho 
= \frac{\abs{\rho^-}^4\rho^+\wedge\rho^++\abs{\rho^+}^4\rho^-\wedge\rho^-}{u^4}
= -\frac{2\abs{\rho^+}^2\abs{\rho^-}^2}{u^3}\dvol.
$$
This proves equation~\eqref{eq:Theta2} and part~(iii).  

To prove part~(iv) choose a smooth path $\R\to\sS_a:t\mapsto\rho_t$ 
such that $\rho_0=\rho$ and $\p_t\rho_t|_{t=0}=\rhohat$. Define
$$
u_t := \frac{\rho_t\wedge\rho_t}{2\dvol},\qquad
\uhat := \left.\frac{\p}{\p t}\right|_{t=0}u_t 
= \frac{\rho\wedge\rhohat}{\dvol},\qquad
\Thetahat := \left.\frac{\p}{\p t}\right|_{t=0}\Theta^{\rho_t}.
$$
Then, by part~(iii) of Proposition~\ref{prop:grho}, we have
$$
R^\rho\rhohat = \rhohat - \frac{\rhohat\wedge\rho}{\dvol_\rho}\rho
= \rhohat - \frac{\uhat}{u}\rho.
$$
Hence
\begin{equation}\label{eq:*rhohat}
\begin{split}
*^\rho\rhohat
&=
R^\rho*R^\rho\rhohat \\
&= 
*R^\rho\rhohat - \frac{\rho\wedge *R^\rho\rhohat}{\dvol_\rho}\rho \\
&= 
*\rhohat
- \frac{\uhat}{u}{*\rho}
- \frac{\rho\wedge *\rhohat}{\dvol_\rho}\rho
+ \frac{\abs{\rho}^2\uhat}{u^2}\rho.
\end{split}
\end{equation}
This implies
\begin{eqnarray*}
\Thetahat
&=&
\left.\frac{\p}{\p t}\right|_{t=0}\left(
\frac{*\rho_t}{u_t} -\frac12\Abs{\frac{\rho_t}{u_t}}^2\rho_t\right) \\
&=&
\frac{*\rhohat}{u} 
- \frac{\uhat}{u}\frac{*\rho}{u}
- \frac{\rho\wedge *\rhohat}{u^2\,\dvol}\rho
+ \frac{\Abs{\rho}^2\uhat}{u^3}\rho
- \frac12\Abs{\frac{\rho}{u}}^2\rhohat \\
&=&
\frac{*^\rho\rhohat}{u} 
-  \frac12\Abs{\frac{\rho}{u}}^2\rhohat \\
&=&
\frac{\rhohat+*^\rho\rhohat}{u} 
-  \Abs{\frac{\rho^+}{u}}^2\rhohat.
\end{eqnarray*}
Here the third step follows from~\eqref{eq:*rhohat} 
and the last step uses the identities
$\abs{\rho}^2=\abs{\rho^+}^2+\abs{\rho^-}^2$
and $2u=\abs{\rho^+}^2-\abs{\rho^-}^2$ 
in~\eqref{eq:urho}. This proves part~(iv).

Equation~\eqref{eq:Theta4} follows from~\eqref{eq:Theta1} 
and the identities 
$$
\frac{\rho^+}{u} = \frac12\sum_iK^\rho_i\om_i,\qquad
\Abs{\frac{\rho^+}{u}}^2 = \frac12\sum_i(K^\rho_i)^2.
$$
To prove~\eqref{eq:Theta5}, define 
$$
\om_i^\rho := \om_i-K^\rho_i\rho,\qquad i=1,2,3. 
$$
Then 
$
\om_i^\rho(\cdot,J^\rho_i\cdot)=g^\rho
$ 
by part~(iv) of Proposition~\ref{prop:grho} and hence 
$$
*^\rho(\lambda\circ J^\rho_i) = \lambda\wedge\om^\rho_i
$$
for all $\lambda\in\Om^1(M)$ and all $i$. 
Take $\lambda=dK^\rho_i$ to obtain 
$$
*^\rho(dK^\rho_i\circ J^\rho_i)
= dK^\rho_i\wedge\om^\rho_i
= d\left(K_i^\rho\om_i-\frac12(K^\rho_i)^2\rho\right).
$$
Take the sum over all $i$ and use 
equation~\eqref{eq:Theta4} to obtain~\eqref{eq:Theta5}. 
This proves part~(v) and Proposition~\ref{prop:Theta}.
\end{proof}

\begin{proposition}[{\bf The Gradient of the Energy}]\label{prop:DGF}
\ 

\noindent{\bf (i)}
The differential of the energy functional $\sE:\sS_a\to\R$
at $\rho\in\sS_a$ and its gradient with respect to the 
Donaldson metric are given by 
\begin{equation}\label{eq:dE}
\begin{split}
\delta\sE(\rho)\rhohat 
&= \int_M\Theta^\rho\wedge\rhohat, \\
\grad\sE(\rho)
&= -d*^\rho d\Theta^\rho.
\end{split}
\end{equation}

\smallskip\noindent{\bf (ii)}
Assume the hyperK\"ahler case. Then 
\begin{equation}\label{eq:dE-HK}
\begin{split}
\sE(\rho) 
&= \frac12\int_M\sum_{i=1}^3\abs{K^\rho_i}^2\dvol_\rho,
\qquad K^\rho_i := \frac{\om_i\wedge\rho}{\dvol_\rho}, \\
\delta\sE(\rho)\rhohat 
&= \int_M\sum_{i=1}^3\left(K_i^\rho\om_i-\frac12(K^\rho_i)^2\rho\right)
\wedge\rhohat, \\
\grad\sE(\rho) 
&= \sum_{i=1}^3 d\left(dK_i^\rho\circ J_i^\rho\right),\qquad
\rho(J^\rho_i\cdot,\cdot):=\rho(\cdot,J_i\cdot).
\end{split}
\end{equation}
\end{proposition}

\begin{proof}
Let $\rho\in\sS_a$ and define $u:=\frac{\dvol_\rho}{\dvol}$. 
Then, by equation~\eqref{eq:urho},
$$
\sE(\rho) - \Vol(M) = \int_M
\frac{\abs{\rho^+}^2+\abs{\rho^-}^2}
{\abs{\rho^+}^2-\abs{\rho^-}^2}\,\dvol
= \int_M\frac{\rho\wedge *\rho}{2u}.
$$
Choose a path $\R\to\sS_a:t\mapsto\rho_t$ 
and define
$$
u_t := \frac{\rho_t\wedge\rho_t}{2\dvol},\qquad
\uhat_t := \p_tu_t = \frac{\rho_t\wedge\rhohat_t}{\dvol},\qquad
\rhohat_t := \p_t\rho_t.
$$
Then
\begin{eqnarray*}
\frac{d}{dt}\sE(\rho_t)
&=&
\frac{d}{dt}\int_M\frac{\rho_t\wedge *\rho_t}{2u_t} 
=
\int_M\frac{\rho_t\wedge *\rhohat_t}{u_t}
- \int_M\frac12\Abs{\frac{\rho_t}{u_t}}^2\,\uhat_t\dvol \\
&=&
\int_M\frac{*\rho_t\wedge \rhohat_t}{u_t}
- \int_M\frac12\Abs{\frac{\rho_t}{u_t}}^2\,\rho_t\wedge\rhohat_t 
=
\int_M\Theta^{\rho_t}\wedge\rhohat_t.
\end{eqnarray*}
This proves the formula for $\delta\sE(\rho)$.  
Now let $\rhohat\in T_\rho\sS_a$ and choose 
$\lambda\in\Om^1(M)$ such that $d\lambda=\rhohat$ 
and $*^\rho\lambda\in\im\,d$.  Then 
$$
\delta\sE(\rho)\rhohat
= \int_M\Theta^\rho\wedge d\lambda
= - \int_M d\Theta^\rho\wedge \lambda
= \INNER{-d*^\rho d\Theta^\rho}{\rhohat}_\rho.
$$
This proves part~(i). In part~(ii) the first equation 
in~\eqref{eq:dE-HK} follows from Proposition~\ref{prop:E}, 
the second equation follows from~\eqref{eq:dE} and~\eqref{eq:Theta4}
and the third equation follows from~\eqref{eq:dE} and~\eqref{eq:Theta5}. 
This proves Proposition~\ref{prop:DGF}.
\end{proof}

\bigbreak

By part~(i) of Proposition~\ref{prop:DGF} a smooth path $\R\to\sS_a:t\mapsto\rho_t$ 
is a negative gradient flow line of $\sE$ with respect to the Donaldson metric
if and only if it satisfies the partial differential equation
\begin{equation}\label{eq:DGF}
\p_t\rho_t = d*^{\rho_t} d\Theta_t,\qquad
\Theta_t := \frac{2\rho_t^+}{u_t}
- \Abs{\frac{\rho_t^+}{u_t}}^2\rho_t,\qquad
u_t := \frac{\dvol_{\rho_t}}{\dvol}.
\end{equation}
Equation~\eqref{eq:DGF} is the {\bf Donaldson Geometric Flow}.
By part~(ii) of Proposition~\ref{prop:DGF} it agrees with the 
geometric flow~\eqref{eq:gradientrho} 
in the hyperK\"ahler case.

\begin{corollary}\label{cor:DGF}
{\bf (i)}
A symplectic form $\rho\in\sS_a$ is a critical point 
of the energy functional $\sE:\sS_a\to\R$ in~\eqref{eq:E} 
if and only if the $2$-form $\Theta^\rho$ is closed.

\smallskip\noindent{\bf (ii)}
Suppose $\om\in\sS_a$ is compatible with the background metric~$g$.
Then $\om$ is the unique absolute minimum of the energy
functional $\sE:\sS_a\to\R$. 

\smallskip\noindent{\bf (iii)}
Assume the hyperK\"ahler case. Then $\rho\in\sS_a$ is a 
critical point of the energy functional $\sE:\sS_a\to\R$
if and only if $\sum_{i=1}^3 dK^\rho_i\circ J^\rho_i=0$.
\end{corollary}

\begin{proof}
Part~(i) follows from equation~\eqref{eq:dE}
in Proposition~\ref{prop:DGF}.  To prove~(ii) observe that 
a symplectic form $\om\in\sS_a$ is compatible with the metric~$g$ 
if and only if it is self-dual. Moreover, every self-dual symplectic 
form is harmonic and the class $a$ has a unique harmonic 
representative.  Since
$$
\frac12\sE(\rho) = \int_M\frac{\abs{\rho^+}^2}
{\abs{\rho^+}^2-\abs{\rho^-}^2}\,\dvol
\ge \int_M\dvol =: \Vol(M)
$$
for all $\rho\in\sS_a$, with equality if and only if $\rho^-=0$, this proves~(ii).
Part~(iii) follows from~(i) and equation~\eqref{eq:Theta5}
in Proposition~\ref{prop:Theta}. 
\end{proof}

The next proposition is an observation of Donaldson~\cite{DON1}
which asserts that the energy controls the $L^1$ norm of $\rho$. 

\begin{proposition}[{\bf Donaldson's $L^1$ Estimate}]\label{prop:DGFL1}
Every $\rho\in\sS_a$ satisfies 
\begin{equation}\label{eq:energyL1}
\Norm{\rho}_{L^1}\le \sqrt{c(\sE(\rho)-\Vol(M))},\qquad
c:=\int_M\rho\wedge\rho=\inner{a^2}{[M]}.
\end{equation}
\end{proposition}

\begin{proof}
By the Cauchy--Schwarz inequality,
\begin{eqnarray*}
\left(\int_M\abs{\rho}\,\dvol\right)^2
&\le&
\left(\int_M\left(\abs{\rho^+}^2-\abs{\rho^-}^2\right)\,\dvol\right)
\int_M\frac{\abs{\rho}^2}{\abs{\rho^+}^2-\abs{\rho^-}^2}\,\dvol \\
&=&
\left(\int_M\rho\wedge\rho\right)
\int_M\frac{\abs{\rho^+}^2+\abs{\rho^-}^2}{\abs{\rho^+}^2-\abs{\rho^-}^2}\,\dvol \\
&=&
c(\sE(\rho)-\Vol(M)).
\end{eqnarray*}
This proves Proposition~\ref{prop:DGFL1}.
\end{proof}

\bigbreak

\begin{remark}\label{rmk:DGF}\rm
{\bf (i)}
Donaldson's conjectural program involves a proof 
of longtime existence for all initial conditions,
a proof that solutions cannot escape to infinity,
and a proof that higher critical points can be {\it bypassed},
i.e.\ that they cannot be local minima.  In those cases 
where this program can be carried out it would then follow
that the space $\sS_a$ is connected, which is an open question
for all closed four-manifolds $M$ and all cohomology classes
$a$ that can be represented by symplectic forms (see~\cite{SAL}).
Short time existence and regularity, as well as long time existence 
for initial conditions sufficiently close to the absolute minimum, 
are established in~\cite{KROM}. 

\smallskip\noindent{\bf (ii)}
In many situations (including certain K\"ahler 
classes on the K3-surface) a theorem of 
Seidel~\cite{SEIDEL1,SEIDEL2,SEIDEL3} asserts the 
existence of symplectomorphisms of $(M,\om)$ that 
are smoothly, but not symplectically, 
isotopic to the identity.  This implies the 
existence of noncontractible loops in $\sS_a$. 
Hence, if the analytic difficulties in Donaldson's 
geometric flow approach can be settled, it would follow
that in these cases the energy functional $\sE:\sS_a\to\R$
must have critical points of index one, assuming
that they are nondegenerate. Many other examples of nontrivial 
cohomology classes in $\sS_a$ of all degrees were found 
by Kronheimer~\cite{KRONHEIMER} using Seiberg--Witten theory.

\smallskip\noindent{\bf (iii)}
By an observation of Donaldson~\cite{DON1} higher critical points of $\sE$ 
(not equal to the absolute minimum) cannot be strictly stable in
the hyperK\"ahler case.  We include a proof of this 
result in Section~\ref{sec:HESSIAN} (Theorem~\ref{thm:DON}).
\end{remark}

\begin{proposition}\label{prop:CP2}
Let $M=\CP^2$ be the complex projective plane 
with its standard K\"ahler metric, 
let $\om_\FS$ be the Fubini--Study form, 
and define 
$$
a:=[\om_\FS]\in H^2(M;\R).  
$$
Then $\om_\FS$ is the only critical point, and the absolute minimum, 
of the energy functional $\sE:\sS_a\to\R$ in~\eqref{eq:E}.
\end{proposition}

\begin{proof}
That $\om_\FS$ is the unique absolute minimum of $\sE$
follows from part~(ii) of Corollary~\ref{cor:DGF}. 
Now let $\rho\in\sS_a$ be any critical point of $\sE$. 
Then $\Theta^\rho$ is closed by part~(i) of 
Proposition~\ref{prop:DGF} and $\Theta^\rho\wedge\rho=0$ 
by part~(i) of Proposition~\ref{prop:Theta}. 
Since $H^2(M;\R)$ is one-dimensional, 
this implies that $\Theta^\rho$ is exact. 
Hence it follows from part~(iii) 
of Proposition~\ref{prop:Theta} that
$$
0 
= \int_{\CP^2}\Theta^\rho\wedge\Theta^\rho
= - \int_{\CP^2}\frac{2\abs{\rho^+}^2\abs{\rho^-}^2}{u^3}\,\dvol_\FS,\qquad
u := \frac{\dvol_\rho}{\dvol_\FS}.
$$
This shows that $\rho^-=0$.  Thus $\rho$ is self-dual 
and hence is harmonic. Since $[\rho]=a=[\om_\FS]$ and $\om_\FS$
is also a harmonic $2$-form it follows that $\rho=\om_\FS$. 
\end{proof}


\section{The Hessian}\label{sec:HESSIAN}

The infinite-dimensional {\it manifold} $\sS_a$ 
is an open set in an affine space. Hence the Hessian of $\sE$ 
is well defined for every $\rho\in\sS_a$ as the second derivative 
${\sH_\rho(\rhohat):=\frac{d^2}{dt^2}|_{t=0}\sE(\rho_t)}$
along a curve $\R\to\sS_a:t\to\rho_t$ satisfying $\rho_0=\rho$,
$\frac{d}{dt}|_{t=0}\rho_t=\rhohat$, and 
$\frac{d^2}{dt^2}|_{t=0}\rho_t=0$.

\begin{theorem}\label{thm:HESSIAN}
Let $\rho\in\sS_a$.  Then the following holds.

\smallskip\noindent{\bf (i)}
The Hessian of $\sE$ at $\rho$ is the quadratic form
$\sH_\rho:T_\rho\sS_a\to\R$ given by 
\begin{equation}\label{eq:HESSIAN1}
\sH_\rho(\rhohat)
= \int_M\Thetahat\wedge\rhohat,\qquad
\Thetahat := \frac{\rhohat + *^\rho\rhohat}{u}
-\Abs{\frac{\rho^+}{u}}^2\rhohat.
\end{equation}
As a linear operator the Hessian is the map
$T_\rho\sS_a\to T_\rho\sS_a:\rhohat\mapsto - d*^\rho d\Thetahat$. 

\smallskip\noindent{\bf (ii)}
Assume the hyperK\"ahler case and define
\begin{equation}\label{eq:K}
K_i:=K_i^\rho = \frac{\om_i\wedge\rho}{\dvol_\rho},\qquad
\om_i^\rho:=\om_i-K_i\rho.
\end{equation}
Let $\rhohat\in T_\rho\sS_a$, choose ${X\in\Vect(M)}$ 
such that $-d\iota(X)\rho=\rhohat$, 
define 
\begin{equation}\label{eq:HKhat}
\Khat_i := \frac{(\om_i-K_i\rho)\wedge\rhohat}{\dvol_\rho},\qquad
\Hhat_i := \frac{(d\iota(X)\om_i)\wedge\rho}{\dvol_\rho},
\end{equation}
and let~$\Thetahat$ be as in~\eqref{eq:HESSIAN1}.
Then 
\begin{equation}\label{eq:Thetahat}
\Thetahat=\sum_{i=1}^3\Khat_i\om_i^\rho-\frac12\sum_{i=1}^3K_i^2\rhohat,
\end{equation}
\begin{equation}\label{eq:HESSIAN2}
\sH_\rho(\rhohat)
= \int_M\sum_i\left(\Khat_i^2\dvol_\rho 
- \frac12K_i^2\rhohat\wedge\rhohat\right).
\end{equation}
Moreover, if $\rho$ is a critical point of $\sE$, then
\begin{equation}\label{eq:HESSIAN3}
\begin{split}
\sH_\rho(\rhohat)
&= 
\int_M\sum_{i=1}^3\om_i(X,X_{\Hhat_i}+\Nabla{X_{K_i}}X)\,\dvol_\rho \\
&=
\int_M\sum_{i=1}^3\left(
\Hhat_i^2\dvol_\rho + \om_i(X,\Nabla{X_{K_i}}X)
\right)\,\dvol_\rho.
\end{split}
\end{equation}
Here $\nabla$ denotes the Levi-Civita connection of the 
hyperK\"ahler metric and $X_F$ denotes the Hamiltonian 
vector field of a function $F:M\to\R$ with respect 
to $\rho$, i.e.\ $\iota(X_F)\rho=dF$. 
\end{theorem}

\begin{proof}
See page~\pageref{proof:HESSIAN}.
\end{proof}

\bigbreak

In~\cite{KROM} it is shown that, for every $\rho\in\sS_a$, 
the quadratic form in~\eqref{eq:HESSIAN3} is the covariant Hessian 
of $\sE$ with respect to the Donaldson metric in Definition~\ref{def:DONmet}.
The proof of Theorem~\ref{thm:HESSIAN} relies on the following two lemmas. 

\begin{lemma}\label{le:HESSIAN1}
Let $\rho$ and $\om$ be symplectic forms on $M$ and define 
$$
K:=\frac{\om\wedge\rho}{\dvol_\rho},\qquad
\om^\rho := \om-K\rho.
$$
Then, for every vector field $X\in\Vect(M)$,
\begin{equation}\label{eq:omrhoX}
(\iota(X)\om)\wedge\rho + \om^\rho\wedge\iota(X)\rho = 0,
\end{equation}
\begin{equation}\label{eq:KomrhoX}
\frac{(d\iota(X)\om)\wedge\rho}{\dvol_\rho} 
+ \frac{\om^\rho\wedge d\iota(X)\rho}{\dvol_\rho}
= \cL_XK.
\end{equation}
If $\om$ is self-dual and $J$ is the almost complex 
structure such that $\om(\cdot,J\cdot)$
is the background Riemannian metric on $M$, then
\begin{equation}\label{eq:omrhoX*}
(\iota(X)\om)\wedge\rho
= -*^\rho\iota(JX)\rho.
\end{equation}
\end{lemma}

\begin{proof}
Equation~\eqref{eq:omrhoX} follows by direct computation, i.e.\ 
\begin{eqnarray*}
\bigl(\iota(X)\om\bigr)\wedge\rho
&=&
\iota(X)\bigl(\om\wedge\rho\bigr) - \om\wedge\iota(X)\rho \\
&=&
K\iota(X)\dvol_\rho - \om\wedge\iota(X)\rho \\
&=&
-\bigl(\om-K\rho\bigr)\wedge\iota(X)\rho.
\end{eqnarray*}
Now differentiate equation~\eqref{eq:omrhoX}
and use the identity $d\om^\rho=-dK\wedge\rho$
to obtain
\begin{eqnarray*}
0 
&=&
d\bigl(\bigl(\iota(X)\om\bigr)\wedge\rho
+ \om^\rho\wedge\iota(X)\rho\bigr) \\
&=&
\bigl(d\iota(X)\om\bigr)\wedge\rho
+ \om^\rho\wedge d\iota(X)\rho
- dK\wedge\rho\wedge\iota(X)\rho \\
&=&
\bigl(d\iota(X)\om\bigr)\wedge\rho
+ \om^\rho\wedge d\iota(X)\rho
- dK\wedge\iota(X)\dvol_\rho \\
&=&
\bigl(d\iota(X)\om\bigr)\wedge\rho
+ \om^\rho\wedge d\iota(X)\rho
- (\iota(X)dK)\dvol_\rho.
\end{eqnarray*}
This proves~\eqref{eq:KomrhoX}.
Now suppose $\om$ is compatible with the almost complex structure $J$
and $\om(\cdot,J\cdot)$ is the background Riemannian metric.
Define the almost complex structure $J^\rho$ by 
$\rho(J^\rho\cdot,\cdot):=\rho(\cdot,J\cdot)$.
Then $g^\rho=\om^\rho(\cdot,J^\rho\cdot)$ by part~(iv) of
Theorem~\ref{thm:FOUR}. Hence it follows 
from~\eqref{eq:omrhoX} and Lemma~\ref{le:gomJ} that
$$
(\iota(X)\om)\wedge\rho
= - \om^\rho\wedge\iota(X)\rho
= -*^\rho(\iota(X)\rho\circ J^\rho)
= -*^\rho\iota(JX)\rho.
$$
This proves equation~\eqref{eq:omrhoX*} 
and Lemma~\ref{le:HESSIAN1}
\end{proof}

\bigbreak

The identities in Lemma~\ref{le:HESSIAN1} are needed
to establish the next result which relates equations~\eqref{eq:HESSIAN2}
and~\eqref{eq:HESSIAN3} and is a key step in the proof of 
Theorem~\ref{thm:HESSIAN}.  It is shown in~\cite{KROM} that the 
last two integrals in equation~\eqref{eq:HESSIANCOV} below
arise from the Levi-Civita connection of the Donaldson metric in 
Definition~\ref{def:DONmet} on the infinite-dimensional 
manifold $\sS_a$.  They vanish for critical points of $\sE$.  
Both sides of equation~\eqref{eq:HESSIANCOV} agree with the 
covariant Hessian of~$\sE$ at an arbitrary element 
$\rho\in\sS_a$ (see~\cite{KROM}).

\begin{lemma}\label{le:HESSIAN2}
Let $\rho\in\sS_a$, let $\rhohat\in T_\rho\sS_a$
be an exact $2$-form, let $K_i,\Khat_i,\Hhat_i$
be as in equations~\eqref{eq:K} and~\eqref{eq:HKhat} 
in Theorem~\ref{thm:HESSIAN}, and let $X\in\Vect(M)$ 
be any vector field such that $-d\iota(X)\rho=\rhohat$. 
Then
\begin{equation}\label{eq:HESSIANCOV}
\begin{split}
&\int_M \sum_{i=1}^3 \left(\Hhat_i^2
+\om_i\bigl(X,\Nabla{X_{K_i}}X\bigr)\right)\dvol_\rho
= \int_M \sum_{i=1}^3 
\left(\Khat_i^2\dvol_\rho-\frac{1}{2}K_i^2\rhohat^2\right) \\
&
+ \int_M \sum_{i=1}^3 (\iota(X_{K_i})\om_i)\wedge(\iota(X)\rho)\wedge\rhohat 
+ \int_M \sum_{i=1}^3 \om_i\bigl(X,\Nabla{X}X_{K_i}\bigr)\dvol_\rho.
\end{split}
\end{equation}
Here $\nabla$ is the Levi-Civita connection of the 
K\"ahler metric and $X_{K_i}$ is the Hamiltonian vector field 
of $K_i$ associated to $\rho$ so that $\iota(X_{K_i})\rho=dK_i$.
\end{lemma}

\begin{proof}
Equation~\eqref{eq:HESSIANCOV} can be written in the form
\begin{equation}\label{eq:HESSIANcov}
\begin{split}
\int_M\sum_{i=1}^3 \Hhat_i^2\dvol_\rho
&=
\int_M \sum_{i=1}^3 
\left(\Khat_i^2\dvol_\rho 
- \frac{1}{2}K_i^2\rhohat\wedge\rhohat \right) \\
&\quad + 
\int_M\sum_{i=1}^3(\iota(X_{K_i})\om_i)\wedge(\iota(X)\rho)\wedge\rhohat \\
&\quad+ 
\int_M \sum_{i=1}^3 \om_i\bigl(X,[X_{K_i},X]\bigr)\dvol_\rho.
\end{split}
\end{equation}
To prove this formula we first observe that
\begin{equation}\label{eq:LXKi}
\cL_XK_i = \Hhat_i-\Khat_i
\end{equation}
for $i=1,2,3$ by equation~\eqref{eq:KomrhoX} 
in Lemma~\ref{le:HESSIAN1}. This implies
\begin{equation}\label{eq:HESSIANcov1}
\begin{split}
\int_M\sum_{i=1}^3 \Hhat_i^2\dvol_\rho
&=
\int_M \sum_{i=1}^3\Khat_i^2\dvol_\rho 
- \int_M \sum_{i=1}^3(\cL_XK_i)^2\dvol_\rho \\
&\quad 
+ 2\int_M \sum_{i=1}^3\Hhat_i(\cL_XK_i)\dvol_\rho.
\end{split}
\end{equation}
Define 
\begin{equation}\label{eq:ABCDE}
\begin{split}
A :=&\; -\frac12\int_M\sum_{i=1}^3K_i^2\rhohat\wedge\rhohat, \\
B :=&\; \int_M\sum_{i=1}^3(\iota(X_{K_i})\om_i)
\wedge(\iota(X)\rho)\wedge\rhohat, \\
C :=&\; \int_M\sum_{i=1}^3\om_i\bigl(X,[X_{K_i},X]\bigr)\dvol_\rho,\\
D :=&\; \int_M\sum_{i=1}^3(\cL_XK_i)^2\,\dvol_\rho,\\
E :=&\; \int_M\sum_{i=1}^3\Hhat_i(\cL_XK_i)\,\dvol_\rho.
\end{split}
\end{equation}
Then equation~\eqref{eq:HESSIANcov1} shows 
that~\eqref{eq:HESSIANcov} is equivalent to the identity
$$
A + B + C + D = 2E. 
$$
To prove this, we first observe that
\begin{eqnarray*}
A
&=&
\int_M\sum_{i=1}^3\frac12K_i^2(d\iota(X)\rho)\wedge\rhohat \\
&=&
- \int_M\sum_{i=1}^3 K_idK_i\wedge(\iota(X)\rho)\wedge\rhohat \\
&=&
- \int_M\sum_{i=1}^3 K_i(\iota(X_{K_i})\rho)
\wedge (\iota(X)\rho)\wedge\rhohat \\
&=&
- B + \int_M\sum_{i=1}^3 (\iota(X_{K_i})(\om_i-K_i\rho))
\wedge(\iota(X)\rho)\wedge\rhohat. 
\end{eqnarray*}
Hence $A+B=F+G$, where
\begin{equation}\label{eq:F12}
\begin{split}
F &:=
- \int_M\sum_{i=1}^3 (\om_i-K_i\rho)
\wedge (\iota(X_{K_i})\iota(X)\rho)\wedge\rhohat, \\
G &:=
\int_M\sum_{i=1}^3 (\om_i-K_i\rho)
\wedge (\iota(X)\rho)\wedge(\iota(X_{K_i})\rhohat).
\end{split}
\end{equation}
Since $\iota(X_{K_i})\iota(X)\rho=-\cL_XK_i$
and $(\om_i-K_i\rho)\wedge\rhohat = \Khat_i\dvol_\rho$,
we have
\begin{eqnarray*}
F
&=&
\int_M\sum_{i=1}^3\Khat_i(\cL_XK_i)\dvol_\rho \\
&=&
\int_M\sum_{i=1}^3
\left(\Hhat_i(\cL_XK_i)-(\cL_XK_i)^2\right)
\,\dvol_\rho \\
&=&
E-D.
\end{eqnarray*}
Here we have used equation~\eqref{eq:LXKi}.
To sum up, we have proved that 
$$
A+B+D=D+F+G=E+G.
$$
Thus it remains to prove that $C=E-G$. 
To see this, observe that
$$
\iota(\cL_{X_{K_i}}X)\dvol_\rho
= \rho\wedge\iota(\cL_{X_{K_i}}X)\rho
= \rho\wedge\cL_{X_{K_i}}(\iota(X)\rho) 
$$
and, by Cartan's formula,
$$
\cL_{X_{K_i}}\bigl(\iota(X)\rho\bigr)
= d\iota(X_{K_i})\iota(X)\rho + \iota(X_{K_i})d\iota(X)\rho
= - d(\cL_XK_i) - \iota(X_{K_i})\rhohat.
$$
Since $\om_i\bigl(X,[X_{K_i},X]\bigr)\dvol_\rho
= - (\iota(X)\om_i)\wedge\iota(\cL_{X_{K_i}}X)\dvol_\rho$,
this implies
\begin{equation*}
\om_i\bigl(X,[X_{K_i},X]\bigr)\dvol_\rho
= (\iota(X)\om_i)\wedge\rho\wedge 
\Bigl(d(\cL_XK_i) + \iota(X_{K_i})\rhohat\Bigr) 
\end{equation*}
for $i=1,2,3$.  Integrate over $M$ 
and take the sum over $i$ to obtain
\begin{equation*}
\begin{split}
C
&= 
\int_M\sum_{i=1}^3\om_i\bigl(X,[X_{K_i},X]\bigr)\dvol_\rho \\
&=
\int_M\sum_{i=1}^3(\iota(X)\om_i)\wedge\rho\wedge d(\cL_XK_i)
+ \int_M\sum_{i=1}^3(\iota(X)\om_i)\wedge\rho\wedge\iota(X_{K_i})\rhohat \\
&=
\int_M\sum_{i=1}^3\bigl(d(\iota(X)\om_i)\wedge\rho\bigr)\wedge\cL_XK_i
+ \int_M\sum_{i=1}^3(\iota(X)\om_i)\wedge\rho\wedge\iota(X_{K_i})\rhohat \\
&=
\int_M\sum_{i=1}^3\Hhat_i(\cL_XK_i)\dvol_\rho
- \int_M\sum_{i=1}^3
(\om_i-K_i\rho)\wedge(\iota(X)\rho)
\wedge\iota(X_{K_i})\rhohat \\
&=
E - G.
\end{split}
\end{equation*}
Here the penultimate equation follows from the definition
of $\Hhat_i$ and from equation~\eqref{eq:omrhoX} 
in Lemma~\ref{le:HESSIAN1}.  Thus $A+B+D=E+G$ 
and $C= E-G$, as claimed, and this completes
the proof of Lemma~\ref{le:HESSIAN2}.
\end{proof}

\begin{proof}[Proof of Theorem~\ref{thm:HESSIAN}]\label{proof:HESSIAN}
Consider the map $\sS_a\to\Om^2(M):\rho\mapsto\Theta^\rho$ in~\eqref{eq:Theta}.
By part~(iv) of Proposition~\ref{prop:Theta} its derivative at 
$\rho$ in the direction $\rhohat$ is given by the $2$-form
$\Thetahat$ in~\eqref{eq:Theta3}.  Hence equation~\eqref{eq:HESSIAN1}
for the Hessian follows from the formula~\eqref{eq:dE} for 
the differential of the energy functional $\sE$ in~\eqref{eq:E}. 

Next we prove equation~\eqref{eq:Thetahat} and~\eqref{eq:HESSIAN2}
in the hyperK\"ahler case. The $2$-forms 
$
\om_i^\rho=\om_i-K_i\rho=R^\rho\om_i
$ 
span the space of self-dual $2$-forms with respect to $g^\rho$ 
by part~(iii) of Proposition~\ref{prop:grho}.  Hence it follows 
from~\eqref{eq:HKhat} that
$$
\frac{\rhohat+*^\rho\rhohat}{u}
= \sum_{i=1}^3\Khat_i\om_i^\rho,\qquad
\Abs{\frac{\rho^+}{u}}^2 = \frac12\sum_{i=1}^3K_i^2.
$$
This proves~\eqref{eq:Thetahat} and~\eqref{eq:HESSIAN2}.
Alternatively, these two equations can be derived from the fact
that $\Theta^\rho$ is given by equation~\eqref{eq:Theta4} 
in the hyperk\"ahler case. Namely, choose a smooth path $\rho_t\in\cS_a$ 
such that $\rho_0=\rho$ and $\frac{d}{dt}|_{t=0}\rho_t=\rhohat$. 
Then $\frac{d}{dt}|_{t=0}K_i^{\rho_t}=\Khat_i$.  
Differentiate~\eqref{eq:Theta4} to obtain that 
$\Thetahat=\frac{d}{dt}|_{t=0}\Theta^{\rho_t}$ is given 
by~\eqref{eq:Thetahat} and inserte this formula 
into~\eqref{eq:HESSIAN1} to obtain~\eqref{eq:HESSIAN2}.

Next observe that the two integrals in~\eqref{eq:HESSIAN3}
agree because
\begin{eqnarray*}
\int_M\Hhat_i^2\,\dvol_\rho
&=& 
\int_M\Hhat_i d(\iota(X)\om_i)\wedge\rho 
=
\int_M\iota(X)\om_i\wedge d\Hhat_i\wedge\rho \\
&=&
\int_M\iota(X)\om_i\wedge\iota(X_{\Hhat_i})\dvol_\rho 
=
\int_M\om_i(X,X_{\Hhat_i})\dvol_\rho.
\end{eqnarray*}
Here the first equation follows from the definition 
of the function $\Hhat_i$ in~\eqref{eq:HKhat} and the third 
equation follows from the fact that $X_{\Hhat_i}$ is its
Hamiltonian vector field with respect to $\rho$.

Now suppose that $\rho\in\sS_a$ is a critical point of 
the energy functional $\sE$ in~\eqref{eq:E}. Then
$$
\sum_{i=1}^3\rho(J_iX_{K_i},\cdot) 
= \sum_{i=1}^3\rho(X_{K_i},J^\rho_i\cdot)
= \sum_{i=1}^3 dK_i\circ J_i^\rho 
= 0
$$
by part~(iii) of Corollary~\ref{cor:DGF}.  Hence
$\sum_{i=1}^3J_iX_{K_i}=0$ and this implies
\begin{equation}\label{eq:crit}
\sum_{i=1}^3J_i\Nabla{X}X_{K_i}=0,\qquad 
\sum_{i=1}^3\iota(X_{K_i})\om_i=0.
\end{equation}
Thus the last two integrals in equation~\eqref{eq:HESSIANCOV}
vanish and so the right hand side of equation~\eqref{eq:HESSIAN2}
agrees with the right hand side of equation~\eqref{eq:HESSIAN3}
by Lemma~\ref{le:HESSIAN2}. This proves Theorem~\ref{thm:HESSIAN}.
\end{proof}

\begin{corollary}\label{cor:MIN}
If $\rho=\om\in\sS_a$ is self-dual, then
$$
\sH_\om(\rhohat)=\int_M\abs{\rhohat}^2\,\dvol
$$
for all $\rhohat\in T_\om\sS$. 
\end{corollary}

\begin{proof}
This follows from~\eqref{eq:HESSIAN1} with 
$u=1$, $\abs{\rho^+}^2=\abs{\om}^2=2$, and $g^\rho=g$. 
\end{proof}

\begin{theorem}[{\bf Donaldson}]\label{thm:DON}
Assume the hyperK\"ahler case and $a:=[\om_1]$.
If $\rho\in\sS_a$ is a critical point of $\sE$ and $\rho\ne\om_1$,
then the Hessian $\sH_\rho$ is not positive definite.
\end{theorem}

\begin{proof}
The proof has four steps.

\medskip\noindent{\bf Step~1.}
{\it Let $\rho\in\Om^2_\ndg(M)$ and define 
$J^\rho_i$ by $\rho(J^\rho_i\cdot,\cdot):=\rho(\cdot,J_i\cdot)$
for $i=1,2,3$.  Then the first order differential operator
$D:\Om^1(M)\to\Om^0(M,\R^4)$ defined by
\begin{equation}\label{eq:D}
D\lambda:=\left(\frac{d*^\rho\lambda}{\dvol},
\frac{d*^\rho(\lambda\circ J_1^\rho)}{\dvol},
\frac{d*^\rho(\lambda\circ J_2^\rho)}{\dvol},
\frac{d*^\rho(\lambda\circ J_3^\rho)}{\dvol}
\right)
\end{equation}
for $\lambda\in\Om^1(M)$ is a Fredholm operator 
of Fredholm index $b_1-4$.}

\medskip\noindent
By part~(iv) of Proposition~\ref{prop:grho}
the $2$-forms 
$$
\om^\rho_i:=\om_i-\frac{\om_i\wedge\rho}{\dvol_\rho}\rho,\qquad
i=1,2,3,
$$
form a basis of the space of self-dual $2$-forms 
with respect to $g^\rho$ and they satisfy 
$\om_i^\rho(\cdot,J^\rho_i\cdot)=g^\rho$ for $i=1,2,3$.
Hence, for $\lambda\in\Om^1(M)$, twice the self-dual
part of $d\lambda$ with respect to $g^\rho$ is the $2$-form
$$
d\lambda+*^\rho d\lambda
= \sum_{i=1}^3\frac{\om_i^\rho\wedge d\lambda}{\dvol}\om_i^\rho.
$$
Hence the self-duality operator 
$$
\Om^1(M)\to\Om^0(M)\oplus\Om^{2,+}_{g^\rho}(M):
\lambda\mapsto(d^{*^\rho}\lambda,d\lambda+*^\rho d\lambda)
$$
of $g^\rho$ is isomorphic to the operator 
$D':\Om^1(M)\to\Om^0(M,\R^4)$ given by 
$$
D'\lambda 
:=\left(\frac{d*^\rho\lambda}{\dvol},
\frac{\om_1^\rho\wedge d\lambda}{\dvol},
\frac{\om_2^\rho\wedge d\lambda}{\dvol},
\frac{\om_3^\rho\wedge d\lambda}{\dvol}
\right).
$$
Hence $D'$ is a Fredholm operator of index $b_1-4$.
Since $\om_i^\rho\wedge\lambda = *^\rho(\lambda\circ J^\rho_i)$
by Lemma~\ref{le:gomJ},  we have 
$$
d*^\rho(\lambda\circ J_i^\rho)-\om_i^\rho\wedge d\lambda
= (d\om^\rho_i)\wedge\lambda.
$$
Hence $D-D'$ is a zeroth order 
operator and therefore is a compact operator between the 
appropriate Sobolev completions. Hence $D$ is a Fredholm 
operator of index $b_1-4$. This proves Step~1.

\medskip\noindent{\bf Step~2.}
{\it Let $\rho\in\sS_a\setminus\{\om_1\}$. 
Then at least one of the functions 
$$
K_i^\rho = \frac{\om_i\wedge\rho}{\dvol_\rho},
\qquad i=1,2,3,
$$
is nonconstant.}

\medskip\noindent
Suppose by contradiction that $K_i^\rho$ is constant 
for $i=1,2,3$.  Since $\rho-\om_1$ is exact, we have
$$
\int_MK_i^\rho\dvol_\rho
= \int_M\om_i\wedge\rho
= \int_M\om_i\wedge\om_1
= \left\{\begin{array}{ll}
2\Vol(M),&\mbox{if }i=1,\\
0,&\mbox{if }i=2,3,
\end{array}\right.
$$
and hence $K_1^\rho=2$ and $K_2^\rho=K_3^\rho=0$.
This implies 
$$
u_1=2u,\qquad u_2=u_3=0,\qquad u:=\frac{\dvol_\rho}{\dvol},\qquad
u_i:=\frac{\om_i\wedge\rho}{\dvol}.
$$
Hence it follows from~\eqref{eq:urho} that
\begin{equation}\label{eq:urho-}
\begin{split}
0 
&\le \abs{\rho^-}^2 \\
&= \abs{\rho^+}^2-2u \\
&= \frac12\sum_{i=1}^3u_i^2 - 2u  \\
&= 2u(u-1).
\end{split}
\end{equation}
Hence $u\ge 1$ and
$$
\int_Mu\dvol =\int_M\dvol_\rho=\int_M\dvol=\Vol(M).
$$
This shows that $u\equiv1$, hence $\rho^-=0$ 
by~\eqref{eq:urho-}, and therefore $\rho=\om_1$. 
This proves Step~2. 

\medskip\noindent{\bf Step~3.}
{\it Let $\rho\in\sS_a\setminus\{\om_1\}$ be a 
critical point of $\sE$. Then there exists a 
$1$-form $\lambda\in\Om^1(M)$ such that
\begin{equation}\label{eq:d*la}
d*^\rho\lambda = 0,\qquad
d*^\rho(\lambda\circ J^\rho_j)=0,\qquad j=1,2,3
\end{equation}
and the exact $2$-forms 
\begin{equation}\label{eq:dla}
\rhohat_0:=d\lambda,\qquad
\rhohat_j:=d(\lambda\circ J^\rho_j),\qquad j=1,2,3.
\end{equation}
are linearly independent.}

\medskip\noindent
By part~(iii) of Corollary~\ref{cor:DGF}, we have 
$$
\sum_{i=1}^3dK_i^\rho\circ J^\rho_i=0.
$$ 
Hence the function 
$$
h:=(0,K_1^\rho,K_2^\rho,K_3^\rho):M\to\R^4\
$$ 
is $L^2$ orthogonal to the image of the operator 
$$
D:\Om^1(M)\to\Om^0(M,\R^4)
$$ 
in Step~1, i.e.\ 
$$
\inner{h}{D\lambda}_{L^2}
= - \int_M\sum_{i=1}^3 dK_i^\rho\wedge *^\rho(\lambda\circ J^\rho_i)
= \int_M \sum_{i=1}^3(dK_i^\rho\circ J^\rho_i)\wedge *^\rho\lambda
= 0
$$
for all $\lambda\in\Om^1(M)$. Since $h$ is nonconstant by Step~2, 
this shows that the cokernel of $D$ has dimension greater than four.
Since $D$ is a Fredholm operator of index $b_1-4$ by Step~1,
its kernel has dimension greater than~$b_1$. The kernel 
of $D$ is a quaternionic vector space and each nonzero element 
$\lambda\in\ker\,D$ determines a four-dimensional quaternionic
subspace 
$$
V_\lambda := \mathrm{span}\left\{\lambda,\lambda\circ J^\rho_1,
\lambda\circ J^\rho_2,\lambda\circ J^\rho_3\right\}\subset\ker\,D.
$$
Denote by 
$$
H^1_{g^\rho}(M):=\left\{\lambda\in\Om^1(M)\,|\,
d\lambda=0,\,d*^\rho\lambda=0\right\}
$$
the space of harmonic $1$-forms with respect to $g^\rho$. 
This space has dimension zero when $M$ is a $K3$-surface
and dimension four when $M$ is a four-torus. Since $\ker\,D$
is a quaternionic vector space of dimension $4k$ with $k\ge2$,
it has a four-dimensional quaternionic subspace that is transverse
to $H^1_{g^\rho}(M)$.  Thus there exists a nonzero element 
$\lambda\in\ker\,D$ such that $V_\lambda\cap H^1_{g^\rho}(M)=0$.
(See Lemma~\ref{le:H}.) This proves Step~3. 

\bigbreak

\medskip\noindent{\bf Step~4.}
{\it Let $\rho\in\sS_a\setminus\{\om_1\}$ be a 
critical point of $\sE$ and let $\lambda\in\Om^1(M)$ 
and~$\rhohat_j$ for $j=0,1,2,3$ be as in Step~3. 
Then} 
$$
\sum_{j=0}^3\sH_\rho(\rhohat_j)=0.
$$

\medskip\noindent
Choose $X\in\Vect(M)$ such that $\iota(X)\rho=-\lambda$.
Then
$$
\rhohat_0 = -d\iota(X)\rho,\qquad
\rhohat_j = -d(\iota(X)\rho\circ J^\rho_j) 
= -d\iota(J_jX)\rho,\qquad j=1,2,3.
$$
For $i,j=1,2,3$ define 
$$
\Hhat_{0i} := \frac{(d\iota(X)\om_i)\wedge\rho}{\dvol_\rho},\qquad
\Hhat_{ji} := \frac{(d\iota(J_jX)\om_i)\wedge\rho}{\dvol_\rho}.
$$
By equation~\eqref{eq:omrhoX} in Lemma~\ref{le:HESSIAN1}, we have 
$$
(d\iota(Y)\om_i)\wedge\rho
= -d(\om_i^\rho\wedge\iota(Y)\rho)
= -d*^\rho((\iota(Y)\rho)\circ J^\rho_i)
$$
for every vector field $Y\in\Vect(M)$. 
Apply this formula to the vector fields $Y=X$ 
and $Y=J_jX$ and use Step~3 to obtain 
$\Hhat_{ji}=0$ for $j=0,1,2,3$ and $i=1,2,3$.  
Hence, by equation~\eqref{eq:HESSIAN3} 
in Theorem~\ref{thm:HESSIAN}, we have 
\begin{equation*}
\begin{split}
\cH_\rho(\rhohat_0) 
&=\sum_{i=1}^3\int_M\om_i(X,\Nabla{X_{K_i^\rho}}X)\dvol_\rho,\\
\cH_\rho(\rhohat_j)
&=\sum_{i=1}^3\int_M\om_i(J_jX,\Nabla{X_{K_i^\rho}}(J_jX))\dvol_\rho
\end{split}
\end{equation*}
for $j=1,2,3$.  Hence
\begin{equation*}
\begin{split}
\sum_{j=1}^3\cH_\rho(\rhohat_j)
&=
\sum_{i,j=1}^3\int_M\inner{X}{J_jJ_iJ_j\Nabla{X_{K_i^\rho}}X}\dvol_\rho \\
&=
\sum_{i=1}^3\int_M\inner{X}{J_i\Nabla{X_{K_i^\rho}}X}\dvol_\rho \\
&=
-\sH_\rho(\rhohat_0)
\end{split}
\end{equation*}
This proves Step~4 and Theorem~\ref{thm:DON}.
\end{proof}

Theorem~\ref{thm:DON} is an infinite-dimensional analogue of a general 
observation about finite-dimensional hyperK\"ahler moment maps.
Let $(M,\om_i,J_i)$ be a hyperK\"ahler manifold,
let~$\G$ be a compact Lie group that acts on~$M$ by hyperK\"ahler
isometries, and for $x\in M$ let $L_x:\fg\to T_xM$ denote the infinitesimal
action of the Lie algebra $\fg=\Lie(\G)$.  Suppose $\fg$ is equipped
with an invariant inner product and the group action is Hamiltonian for 
each~$\om_i$. For $i=1,2,3$ let $\mu_i:M\to\fg$ be an equivariant moment map 
so that $\inner{d\mu_i(x)\xhat}{\xi}=\om_i(L_x\xi,\xhat)$ for $\xi\in\fg$ and
$\xhat\in T_xM$.  Then the gradient of the function
$\cE:=\frac12\sum_i\Norm{\mu_i}^2$ is given by 
$\grad\cE(x)=\sum_iJ_iL_x\mu_i(x)$. Assume $\dim M\ge4\dim\G$
and let $x\in M$ be a critical point of $\cE$ such that $\cE(x)\ne0$.
Then the linear map $\fg^4\to T_xM:(\xi_0,\xi_1,\xi_2,\xi_3)\mapsto L_x\xi_0+\sum_iJ_iL_x\xi_i$
is not injective and hence not surjective. Thus there exists a vector $\xhat\in T_xM$
such that $L_x^*\xhat=0$ and $L_x^*J_i\xhat=0$ for all $i$.
Denote by $\cH_x:T_xM\to\R$ the Hessian of $\cE$ at $x$.
Then a calculation shows that 
$\cH_x(\xhat)+\sum_i\cH_x(J_i\xhat)=0$.
(See Donaldson~\cite[Proposition~6]{DON1}.)   
In the case at hand it would be interesting to find an 
exact $2$-form~$\rhohat$ such that $\cH_\rho(\rhohat)<0$. 


\appendix

\section{Four-Dimensional Linear Algebra}\label{app:FOUR}

Let $V$ be a $4$-dimensional oriented real vector space
and let $V^*:=\Hom(V,\R)$ be the dual space.  Associated to an
inner product $g:V\times V\to\R$ is the Hodge $*$-operator
$*_g:\Lambda^kV^*\to\Lambda^{4-k}V^*$, the volume form 
${\dvol_g=*_g1\in\Lambda^4V^*}$, and the space 
$$
\Lambda_g^+:=\left\{\om\in\Lambda^2V^*\,\big|\,\om=*_g\om\right\}
$$ 
of self-dual $2$-forms.  
By a well known observation (which we learned from~\cite{DON2}) 
the inner product~$g$ is uniquely determined by $\dvol_g$ 
and $\Lambda^+_g$.  
This is the content of Theorem~\ref{thm:DONfour} below.
Call a linear subspace $\Lambda\subset\Lambda^2V^*$ 
{\bf positive} if the quadratic form 
$\Lambda\times\Lambda\to\R:(\om,\tau)\mapsto\frac{\om\wedge\tau}{\dvol}$
is positive definite for some (and hence every) positive 
volume form $\dvol\in\Lambda^4V^*$.  Denote by 
$\cG(V)$ the space of all inner products $g:V\times V\to\R$,
by $\cS(V)$ the space of $2$-forms $\rho\in\Lambda^2V^*$ 
such that $\rho\wedge\rho>0$, and by $\cJ(V)$
the set of linear complex structures $J:V\to V$ that 
are compatible with the orientation.

\begin{theorem}\label{thm:DONfour}
For every positive volume form $\dvol\in\Lambda^4V^*$
and every three-dimensional positive linear subspace 
$\Lambda^+\subset\Lambda^2V^*$ there exists a unique 
inner product $g$ on $V$ such that $\dvol_g=\dvol$
and $\Lambda^+_g=\Lambda^+$.
\end{theorem}

\begin{proof}
See page~\pageref{proof:DONfour}.
\end{proof}

\begin{theorem}\label{thm:FOUR}
Let $g\in\cG(V)$, $\rho\in\cS(V)$, define $u>0$ and $A\in\GL(V)$ by
\begin{equation}\label{eq:uA}
u:= \frac{\dvol_\rho}{\dvol_g},\qquad
\dvol_\rho := \frac{\rho\wedge\rho}{2},\qquad
g(A\cdot,\cdot):=\rho,
\end{equation}
and define the linear map $R:\Lambda^2V^*\to\Lambda^2V^*$ by
\begin{equation}\label{eq:R}
R\om:=\om - \frac{\om\wedge\rho}{\dvol_\rho}\rho
\qquad\mbox{for all }\om\in\Lambda^2V^*.
\end{equation}
Then $R$ is an involution that preserves the exterior product, 
acts as the identity on the orthogonal complement of $\rho$ 
with respect to the exterior product, and $R\rho=-\rho$. 
Moreover, for every $\tg\in\cG(V)$, the following are equivalent.

\smallskip\noindent{\bf (i)}
$\tg(v,w)=u^{-1}g(Av,Aw)$ for all $v,w\in V$.

\smallskip\noindent{\bf (ii)}
$\dvol_\tg=\dvol_g$ and 
$*_\tg\lambda = u^{-1}\rho\wedge *_g(\rho\wedge\lambda)$
for all $\lambda\in V^*$.

\smallskip\noindent{\bf (iii)}
$\dvol_\tg=\dvol_g$ and 
$*_\tg \iota(v)\rho = -\rho\wedge g(v,\cdot)$
for all $v\in V$.

\smallskip\noindent{\bf (iv)}
Suppose $\om\in\cS(V)$ and $J\in\cJ(V)$ satisfy $g=\om(\cdot,J\cdot)$.
Define $\tom\in\Lambda^2V^*$ and $\tJ\in\cJ(V)$ by 
$\tom := R\om$ and $\rho(\tJ\cdot,\cdot):=\rho(\cdot,J\cdot)$.
Then $\tg=\tom(\cdot,\tJ\cdot)$.

\smallskip\noindent{\bf (v)}
$\dvol_\tg=\dvol_g$ and $\Lambda^+_\tg = R\Lambda^+_g$.

\smallskip\noindent{\bf (vi)}
$\dvol_\tg=\dvol_g$ and $*_\tg\om = R*_gR\om$ 
for all $\om\in\Lambda^2V^*$.
\end{theorem}

\begin{proof}
See page~\pageref{proof:FOUR}.
\end{proof}

\bigbreak

The proofs of both theorems are based 
on the following six lemmas.

\begin{lemma}\label{le:ivdvol}
For every $g\in\cG(V)$ and every $v\in V$
$$
*_g\iota(v)\dvol_g = - g(v,\cdot),\qquad
*_gg(v,\cdot) = \iota(v)\dvol_g.
$$
\end{lemma}

\begin{proof}
Direct verification for the standard structures on $V=\R^4$.
\end{proof}

\begin{lemma}\label{le:gomJ}
Let $\om\in\cS(V)$, $g\in\cG(V)$, $J\in\cJ(V)$.
The following are equivalent.

\smallskip\noindent{\bf (i)}
$\om(v,Jw)=g(v,w)$ for all $v,w\in V$.

\smallskip\noindent{\bf (ii)}
$\dvol_\om=\dvol_g$ and 
$*_g(\om\wedge\lambda) = - \lambda\circ J$
for all $\lambda\in V^*$.
\end{lemma}

\begin{proof}
That~(i) implies~(ii) follows by direct verification 
for the standard structures on $V=\C^2$. 
We prove that~(ii) implies~(i). 
Assume $\om,g,J$ satisfy~(ii) and let $v\in V$. 
Then, by Lemma~\ref{le:ivdvol} and~(ii),
$$
g(v,\cdot)
= -*_g\iota(v)\dvol_\om 
= -*_g\bigl(\om\wedge\iota(v)\om\bigr) 
= \iota(v)\om\circ J 
= \om(v,J\cdot).
$$
Hence $\om,g,J$ satisfy~(i).  
This proves Lemma~\ref{le:gomJ}.
\end{proof}

A symplectic form $\om\in\cS(V)$ is called 
{\bf compatible with the inner product $g\in\cG(V)$}
(respectively {\bf compatible with the complex structure $J\in\cJ(V)$})
if there exists a $J\in\cJ(V)$ (respectively a $g\in\cG(V)$)
such that the equivalent conditions~(i) and~(ii) 
in Lemma~\ref{le:gomJ} are satisfied.

\begin{lemma}\label{le:om*}
Let $\om\in\cS(V)$ and $g\in\cG(V)$.
The following are equivalent.

\smallskip\noindent{\bf (i)}
$\om$ is compatible with $g$.

\smallskip\noindent{\bf (ii)}
$\dvol_\om=\dvol_g$ and $\om\in\Lambda^+_g$.
\end{lemma}

\begin{proof}
That~(i) implies~(ii) follows by direct verification 
for the standard structures on $V=\C^2$. 
To prove the converse, consider the standard 
inner product and orientation on the quaternions $V=\H$ 
with coordinates ${x=x_0+\i x_1+\j x_2+\k x_3}$.
Define 
$$
\om_i:=dx_0\wedge dx_i+dx_j\wedge dx_k
$$ 
for $i=1,2,3$ and $i,j,k$ a cyclic permutation of $1,2,3$.
If $\om$ satisfies~(ii), then 
$$
\om=\sum_it_i\om_i,\qquad
t_i\in\R,\qquad \sum_it_i^2=1.
$$ 
Hence $\om$ is compatible with the inner product 
and the complex structure 
$$
J:=t_1\i+t_2\j+t_3\k
$$ 
(acting on $\H$ on the left). This proves Lemma~\ref{le:om*}.
\end{proof}

\begin{lemma}\label{le:det}
Let $\rho\in\cS(V)$ and $g\in\cG(V)$.
If $u$ and $A$ are defined by~\eqref{eq:uA}, then
$$
\det(A)=u^2.
$$
\end{lemma}

\begin{proof}
Assume $V=\R^4$ with the standard inner product 
and standard orientation.  Denote the coordinates 
on $\R^4$ by $x=(x_0,x_1,x_2,x_3)$ and write
$$
\rho = \sum_{i<j}\rho_{ij}dx_i\wedge dx_j,\qquad 
\rho_{ij}+\rho_{ji} = 0.
$$
The nondegeneracy and orientation condition on $\rho$ asserts that
\begin{equation}\label{eq:u2}
u = \rho_{01}\rho_{23} + \rho_{02}\rho_{31} + \rho_{03}\rho_{12} > 0.
\end{equation}
In the standard basis of $\R^4$ the linear
operator $A$ is represented by the matrix 
\begin{equation}\label{eq:A}
A = 
\left(\begin{array}{cccc}
0 & \rho_{01} & \rho_{02} & \rho_{03} \\
-\rho_{01} & 0 & \rho_{12} & -\rho_{31} \\
-\rho_{02} & -\rho_{12} & 0 & \rho_{23} \\
-\rho_{03} & \rho_{31} & -\rho_{23} & 0
\end{array}\right).
\end{equation}
It follows from equation~\eqref{eq:A} that 
\begin{eqnarray*}
\det(A)
&=&
\rho_{01}\det
\left(\begin{array}{ccc}
\rho_{01} & \rho_{02} & \rho_{03} \\
-\rho_{12} & 0 & \rho_{23} \\
\rho_{31} & -\rho_{23} & 0
\end{array}\right) \\
&&
-\, \rho_{02}
\det
\left(\begin{array}{ccc}
\rho_{01} & \rho_{02} & \rho_{03} \\
0 & \rho_{12} & -\rho_{31} \\
\rho_{31} & -\rho_{23} & 0
\end{array}\right) \\
&&
+\, \rho_{03} 
\det
\left(\begin{array}{ccc}
\rho_{01} & \rho_{02} & \rho_{03} \\
0 & \rho_{12} & -\rho_{31} \\
-\rho_{12} & 0 & \rho_{23} \\
\end{array}\right) \\
&=&
\rho_{01}\left(\rho_{02}\rho_{23}\rho_{31} + \rho_{03}\rho_{12}\rho_{23}
+ \rho_{01}\rho_{23}^2\right) \\
&&
+\, \rho_{02}\left(\rho_{03}\rho_{12}\rho_{31} 
+ \rho_{01}\rho_{23}\rho_{31}
+ \rho_{02}\rho_{31}^2\right) \\
&&
+\, \rho_{03}\left(\rho_{01}\rho_{23}\rho_{12} + \rho_{02}\rho_{31}\rho_{12}
+ \rho_{03}\rho_{12}^2\right).
\end{eqnarray*}
Thus $\det(A)=u^2$ 
by~\eqref{eq:u2} 
and this proves Lemma~\ref{le:det}.
\end{proof}

\begin{lemma}\label{le:quaternion}
Let $\om_1,\om_2,\om_3\in\cS(V)$ and 
$J_1,J_2,J_3\in\GL(V)$ such that
\begin{equation}\label{eq:om123}
\om_2(\cdot,J_3\cdot):=\om_1,\qquad
\om_3(\cdot,J_1\cdot):=\om_2,\qquad
\om_1(\cdot,J_2\cdot):=\om_3.
\end{equation}
Then
\begin{equation}\label{eq:omijk}
\om_i(J_jv,w)=\om_i(v,J_jw)=\om_k(v,w)
\end{equation}
for every cyclic permutation $i,j,k$ of $1,2,3$ and all $v,w\in V$.
Moreover, the following are equivalent.

\smallskip\noindent{\bf (i)} 
$\om_i\wedge\om_j = 0$ and
$\om_i\wedge\om_i=\om_j\wedge\om_j$
for $1\le i<j\le3$.

\smallskip\noindent{\bf (ii)} 
$J_i^2=-\one$ and $J_jJ_k=-J_kJ_j=J_i$
for cyclic permutations $i,j,k$ of $1,2,3$.

\medskip\noindent
If these equivalent conditions are satisfied,
then the following holds.

\smallskip\noindent{\bf (a)}
The vectors $v,J_1v,J_2v,J_3v$ form a basis 
of $V$ for every $v\in V\setminus\{0\}$.

\smallskip\noindent{\bf (b)}
$\om_1(v,J_1w)=\om_2(v,J_2w)=\om_3(v,J_3w)$ for $v,w\in V$.

\smallskip\noindent{\bf (c)}
$\om_i(w,J_iv)=\om_i(v,J_iw)$ for $i=1,2,3$ and $v,w\in V$.

\smallskip\noindent{\bf (d)}
$\om_i(v,J_iv)\ne 0$ for $i=1,2,3$ 
and $v\in V\setminus\{0\}$.
\end{lemma}

\bigbreak

\begin{proof}
That~\eqref{eq:om123} implies~\eqref{eq:omijk}
follows from the skew-symmetry of the $\om_i$.

{\bf (i) implies~(ii).}
Since $\iota(J_jv)\om_i=\iota(v)\om_k$,
it follows from~(i) that
$\om_i\wedge\iota(v)\om_i
= \om_k\wedge\iota(v)\om_k 
= \om_k\wedge\iota(J_jv)\om_i 
= -\om_i\wedge\iota(J_jv)\om_k$
for $v\in V$ and every cyclic permutation $i,j,k$ of $1,2,3$.
Hence
\begin{equation}\label{eq:omkji}
\om_k(J_jv,w)=\om_k(v,J_jw)= - \om_i(v,w).
\end{equation}
Second, 
$
\om_2(\cdot,J_3J_2J_1\cdot)
= \om_1(\cdot,J_2J_1\cdot)
= \om_3(\cdot,J_1\cdot)
= \om_2
$,
by equation~\eqref{eq:omijk}, and 
$
\om_2(\cdot,J_1J_2J_3\cdot)
= -\om_3(\cdot,J_2J_3\cdot)
= \om_1(\cdot,J_3\cdot)
= -\om_2
$,
by equation~\eqref{eq:omkji}. Hence
\begin{equation}\label{eq:J123}
J_3J_2J_1=\one=-J_1J_2J_3.
\end{equation}
Third, by~\eqref{eq:omijk} and~\eqref{eq:omkji}, 
$\om_j(\cdot,J_i^2\cdot)=-\om_k(\cdot,J_i\cdot)=-\om_j$
and hence
\begin{equation}\label{eq:quaternion1}
J_1^2=J_2^3=J_3^2=-\one.
\end{equation}
Fourth, $J_2J_1=J_3^{-1}=-J_1J_2$, by~\eqref{eq:J123},
and hence $J_2J_1=-J_3=-J_1J_2$, by~\eqref{eq:quaternion1}.
Multiply this equation by $J_1$ and $J_2$ 
on the left and right to obtain the quaternion relations 
$J_iJ_j=-J_jJ_i=J_k$ for $i,j,k$ cyclic.  
This shows that~(i) implies~(ii).

{\bf (ii) implies~(a).}
Let $v\in V\setminus\{0\}$ and $x_i\in\R$ 
such that $x_0v+\sum_ix_iJ_iv=0$. 
Then 
$$
0 = \left(x_0\one-\sum_{i=1}^3x_iJ_i\right)
\left(x_0v+\sum_{i=1}^3x_iJ_iv\right)
= \left(\sum_{i=0}^3x_i^2\right)v
$$
and hence $x_0=x_1=x_2=x_3=0$. 

{\bf (ii) implies~(b).}
It follows from equation~\eqref{eq:omijk} that,
for $i,j,k$ cyclic,
$\om_i(v,J_iw) = \om_j(J_kv,J_iw) 
= \om_j(v,J_kJ_iw) = \om_j(v,J_jw)$.

{\bf (ii) implies~\eqref{eq:omkji}.}
It follows from equation~\eqref{eq:omijk} that,
for $i,j,k$ cyclic,
$\om_k(J_jv,w)=\om_i(J_jJ_jv,w)=-\om_i(v,w)$.

{\bf (ii) implies~(c).}
It follows from equation~\eqref{eq:omkji} that,
for $i,j,k$ cyclic,
$\om_i(w,J_iv) = \om_i(w,J_jJ_kv) 
= \om_i(J_kJ_jw,v) = \om_i(-J_iw,v) 
= \om_i(v,J_iw)$.

{\bf (ii) implies~(d).}
Fix a nonzero vector $v\in V$.
Then $\om_1(v,v)=0$ and, by~\eqref{eq:omijk} 
and~\eqref{eq:omkji}, $\om_1(v,J_2v)=\om_3(v,v)=0$
and $\om_1(v,J_3v)=-\om_2(v,v)=0$. 
Since $\om_1$ is nondegenerate,
it follows from~(a) that $\om_1(v,J_1v)\ne 0$. 

{\bf (ii) implies~(i).}
Fix a nonzero vector $v\in V$ and define 
$\Phi:\H\to V$ by $\Phi(x):=x_0v+\sum_ix_iJ_iv$.
By~(a) this is an isomorphism.  By~(b) and~(d), 
$$
\lambda:=\om_1(v,J_1v)=\om_2(v,J_2v)=\om_3(v,J_3v)\ne0.
$$
By~\eqref{eq:omijk} and~\eqref{eq:omkji}, we have
$\Phi^*\om_i=\lambda(dx_0\wedge dx_i+dx_j\wedge dx_k)$ for 
$i=1,2,3$ and $i,j,k$ a cyclic permutation of $1,2,3$.
This shows that~(ii) implies~(i).
\end{proof}

\begin{lemma}\label{le:J}
Assume that $J_1,J_2,J_3\in\cJ(V)$ are compatible with $g\in\cG(V)$ 
and satisfy $J_iJ_j+J_jJ_i=0$ for $i\ne j$.   Then $J_3=\pm J_1J_2$.
\end{lemma}

\begin{proof}
Fix a unit vector $v$. Then 
$g(J_iv,J_jv)=g(v,J_jJ_iv)=-g(J_jv,J_iv)$.
Hence $v,J_1v,J_2v,J_3v$ form an 
orthonormal basis of $V$ and $J_1J_2v$ 
is orthogonal to $v,J_1v,J_2v$. 
Hence $J_1J_2v=\pm J_3v$. 
It follows that $J_1J_2=\pm J_3$.
\end{proof} 

\begin{proof}[Proof of Theorem~\ref{thm:DONfour}]
\label{proof:DONfour} {\bf Existence.}  
Fix a basis $\om_1,\om_2,\om_3$ of $\Lambda^+$ such that
$$
\om_i\wedge\om_j=2\delta_{ij}\dvol.
$$
Choose $J_i\in\GL(V)$ such that~\eqref{eq:om123} 
holds. By Lemma~\ref{le:quaternion}, the 
bilinear map 
$$
V\times V\to\R:(v,w)\mapsto\om_i(v,J_iw)
$$
is independent of $i$, symmetric, and definite.
Assume without loss of generality that $\om_i(v,J_iv)>0$
for all $v\in V\setminus\{0\}$.
(Otherwise, replace the triple $J_1,J_2,\om_3$ by $-J_1,-J_2,-\om_3$.)
Then the inner product $g(v,w):=\om_i(v,J_iw)$ 
is compatible with $\om_i$. Hence it follows from Lemma~\ref{le:om*}
that $\dvol_g=\dvol_{\om_i}$ and $\om_i\in\Lambda^+_g$ 
for $i=1,2,3$. Thus $\dvol_g=\dvol$ and $\Lambda^+_g=\Lambda^+$.

\medskip\noindent{\bf Uniqueness.}
Let $\tg\in\cG(V)$ such that $\Lambda_\tg^+=\Lambda^+$ 
and $\dvol_\tg=\dvol$. By Lemma~\ref{le:om*}, 
the symplectic forms $\om_1,\om_2,\om_3$ are compatible 
with $\tg$. Hence there exist complex structures 
$\tJ_1,\tJ_2,\tJ_3\in\cJ(V)$ such that
$$
\om_i(\cdot,\tJ_i\cdot)=\tg(\cdot,\cdot).
$$
Thus
$$
\om_j(\cdot,\tJ_j\tJ_k\cdot) = \tg(\cdot,\tJ_k\cdot)
= - \om_k(\cdot,\cdot) = \om_j(\cdot,J_i\cdot)
$$
by~\eqref{eq:omkji}
and so $\tJ_j\tJ_k=J_i$ for $i,j,k$ cyclic.
Hence
$$
\tJ_j\tJ_k+\tJ_k\tJ_j
=J_i-\tJ_k\tJ_i\tJ_i\tJ_j
=J_i-J_jJ_k=0
$$
for $i,j,k$ cyclic.
By Lemma~\ref{le:J}, 
$$
\tJ_3=\pm\tJ_1\tJ_2=\pm J_3.
$$ 
Since $\om_3(v,J_3v)>0$ and $\om_3(v,\tJ_3v)>0$
for $v\ne 0$, we have $\tJ_3=J_3$.  Hence 
$$
\tg=\om_3(\cdot,\tJ_3\cdot)=\om_3(\cdot,J_3\cdot)=g.
$$
This proves Theorem~\ref{thm:DONfour}.
\end{proof}

\begin{proof}[Proof of Theorem~\ref{thm:FOUR}]
\label{proof:FOUR}
That the linear map $R:\Lambda^2V^*\to\Lambda^2V^*$ 
in~\eqref{eq:R} has the required properties follows by
direct calculation.
 
We prove that~(i) implies~(ii).
By Lemma~\ref{le:det}, $\det(-A^2)=\det(A)^2=u^4$
and hence the inner product $\tg(v,w):=u^{-1}g(Av,Aw)$
has the volume form 
$$
\dvol_\tg = \dvol_g = u^{-1}\dvol_\rho.
$$
Now let $\lambda\in V^*$ and choose $v\in V$ 
such that $\tg(v,\cdot)=\lambda$.  Then,
by Lemma~\ref{le:ivdvol},
\begin{equation}\label{eq:lambdatg}
*_\tg\lambda 
= *_\tg\tg(v,\cdot) 
= \iota(v)\dvol_\tg 
= u^{-1}\iota(v)\dvol_\rho 
= u^{-1}\rho\wedge\iota(v)\rho.
\end{equation}
Since $\iota(v)\rho=g(Av,\cdot)$ it follows also from 
Lemma~\ref{le:ivdvol} that 
\begin{equation*}
\begin{split}
*_g\iota(v)\rho 
&= 
\iota(Av)\dvol_g 
=
u^{-1}\iota(Av)\dvol_\rho 
=
u^{-1}\rho\wedge\iota(Av)\rho \\
&=
u^{-1}\rho\wedge g(A^2v,\cdot) 
=
- \rho\wedge \tg(v,\cdot) 
=
- \rho\wedge\lambda.
\end{split}
\end{equation*}
Thus $\iota(v)\rho=*_g(\rho\wedge\lambda)$ and so
$*_\tg\lambda = u^{-1}\rho\wedge*_g(\rho\wedge\lambda)$
by~\eqref{eq:lambdatg}.
This shows that $\tg$ satisfies~(ii).

We prove that~(ii) implies~(iii).
Assume $\tg$ satisfies~(ii) and let $v\in V$.
Use the equation 
$
u^{-1}(\rho\wedge\iota(v)\rho)
= u^{-1}\iota(v)\dvol_\rho
= \iota(v)\dvol_g
$
to obtain
$$
*_\tg\iota(v)\rho 
=
u^{-1}\rho\wedge*_g(\rho\wedge\iota(v)\rho) 
=
\rho\wedge*_g\iota(v)\dvol_g 
=
- \rho\wedge g(v,\cdot).
$$
Here the last step follows from Lemma~\ref{le:ivdvol}.
This shows that $\tg$ satisfies~(iii).

We prove that~(iii) implies~(iv).
Assume $\tg$ satisfies~(iii). Let $\om\in\cS(V)$ 
and $J\in\cJ(V)$ such that $\om(\cdot,J\cdot)=g$.
Then, by Lemma~\ref{le:gomJ}, 
\begin{equation}\label{eq:compatible}
\dvol_\om=\dvol_g,\qquad
*_g(\om\wedge\lambda) = -\lambda\circ J
\end{equation}
for every $\lambda\in V^*$.  Define $\tom$ and $\tJ$ by
$\tom:=R\om$ and $\rho(\tJ\cdot,\cdot):=\rho(\cdot,J\cdot)$.
Then $\tom\wedge\tom=\om\wedge\om$ and so
$\dvol_\tom = \dvol_\om = \dvol_g = \dvol_\tg$ by~(iii).
Now let $\lambda\in V^*$ and choose $v\in V$ 
such that $\iota(v)\rho=\lambda$.  Then
$$
\lambda\circ\tJ
= \iota(v)\rho\circ\tJ
= \rho(v,\tJ\cdot)
= \rho(Jv,\cdot)
= \iota(Jv)\rho.
$$ 
Abbreviate $K:=\frac{\om\wedge\rho}{\dvol_\rho}$.  Then $\tom=\om-K\rho$
and so
$$
\tom\wedge\lambda
= \om\wedge\iota(v)\rho - K\iota(v)\dvol_\rho
= \om\wedge\iota(v)\rho - \iota(v)(\om\wedge\rho)
= -(\iota(v)\om)\wedge\rho.
$$
By~(iii) this implies
$$
*_\tg(\tom\wedge\lambda)
= -*_\tg(\rho\wedge\iota(v)\om)
= -*_\tg(\rho\wedge g(Jv,\cdot))
= - \iota(Jv)\rho
= - \lambda\circ\tJ.
$$
Hence $\tom(\cdot,\tJ\cdot)=\tg$ by Lemma~\ref{le:gomJ}.
This shows that $\tg$ satisfies~(iv).

\bigbreak

We prove that~(iv)  implies~(v).
Assume $\tg$ satisfies~(iv) and choose a symplectic form 
$\om\in\cS(V)$ that is compatible with $g$.
Then $\tom:=R\om$ is compatible with $\tg$ by~(iv),
and hence 
$$
\dvol_\tg = \dvol_\tom = \dvol_\om = \dvol_g
$$
by Lemma~\ref{le:om*}.  
If $\om\in\Lambda^+_g\setminus\{0\}$, then by Lemma~\ref{le:om*} 
there is a $c>0$ such that $c\om$ is compatible with $g$, 
hence $cR\om$ is compatible with $\tg$ by~(iv), 
and hence $c\tom\in\Lambda^+_\tg$ by Lemma~\ref{le:om*}.
This shows that $R\Lambda^+_g\subset\Lambda^+_\tg$.  
Since $R$ is an involution of $\Lambda^2V^*$, the subspace 
$R\Lambda^+_g$ has dimension three and hence 
agrees with $\Lambda^+_\tg$. This shows that $\tg$ satsfies~(v).

We prove that~(v) implies~(vi).
The map $R:\Lambda^2V^*\to\Lambda^2V^*$ in~\eqref{eq:R}
is an involution and preserves the exterior product, i.e.\ 
$$
R\circ R=\id,\qquad R\om\wedge R\tau =\om\wedge\tau
$$
for all $\om,\tau\in\Lambda^2V^*$.  By~(v) it also satisfies
$$
R\Lambda^+_g = \Lambda^+_\tg.
$$
If $\tau\in\Lambda^-_g$, then $R\tau\wedge R\om=\tau\wedge\om=0$
for all $\om\in\Lambda^+_g$, hence $R\tau\wedge\tom=0$ for every
$\tom\in\Lambda^+_\tg$, and hence $R\tau\in\Lambda^-_\tg$.
Thus $R\Lambda^-_g = \Lambda^-_\tg$. It follows that 
$$
R*_g\om=R\om=*_\tg R\om,\qquad
R*_g\tau=-R\tau=*_\tg R\tau
$$
for all $\om\in\Lambda^+_g$
and all $\tau\in\Lambda^-_g$.
This shows that $R*_g=*_\tg R$ on $\Lambda^2V^*$ 
and hence $\tg$ satisfies~(vi).

We prove that~(vi) implies~(i).
Let $\tg\in\cG(V)$ be any inner product that satisfies~(vi)
and let $h\in\cG(V)$ be the inner product defined by the formula
$$
h(v,w):=u^{-1}g(Av,Aw)
$$ 
in~(i).  Since we have already proved that~(i) implies~(vi), 
the inner products $\tg$ and $h$ both satisfy~(vi).  Thus they 
have the same volume form and the same Hodge $*$-operator
on $2$-forms.  Hence
$$
\dvol_\tg=\dvol_h,\qquad \Lambda^+_\tg=\Lambda^+_h
$$
and so $\tg=h$ by Theorem~\ref{thm:DONfour}.
In other words, every inner product $\tg\in\cG(V)$ that satisfies~(vi) 
is given by $\tg(v,w)=u^{-1}g(Av,Aw)$.  
This completes the proof of Theorem~\ref{thm:FOUR}.
\end{proof}


\section{Quaternionic Subspaces}\label{app:H}

Denote by $\H\cong\R^4$ the quaternions and by 
$\Sp(1)\cong S^3$ the unit quaternions.
For $\lambda\in\H$ define
$V_\lambda := \left\{(x,x\lambda)\,|\,x\in\H\right\}$.
Thus $V_\lambda$ is the unique quaternionic subspace 
of $\H^2$ of real dimension four that contains 
the pair $(1,\lambda)$. 

\begin{lemma}\label{le:H}
Let $W\subset\H^2$ be a real linear
subspace of real dimension $\dim^\R W\le 4$.
Then there exists an element $\lambda\in\H$
such that $V_\lambda\cap W=0$. 
\end{lemma}

\begin{proof}
The proof is a standard transversality argument 
and has two steps.

\medskip\noindent{\bf Step~1.}
{\it Define $f:\Sp(1)\times\H\to\H^2$ by
$f(x,\lambda) := (x,x\lambda)$ for $x\in\Sp(1)$ 
and $\lambda\in\H$. Then $f$ is transverse to 
every real linear subspace of $\H^2$.}

\medskip\noindent
Let $W\subset\H^2$ be a real linear subspace
and let $(x,\lambda)\in\Sp(1)\times\H$ such that
$f(x,\lambda)\in W$.  We must prove that
$\im df(x,\lambda) + W =\H^2$.
To see this, fix any pair $(\xi,\eta)\in\H^2$ 
and define $\xhat := \xi-\inner{\xi}{x}x$ and
$\lahat := x^{-1}(\eta-\xi\lambda)$. Then 
$$
df(x,\lambda)(\xhat,\lahat) - (\xi,\eta)
= (\xhat-\xi,\xhat\lambda+x\lahat-\eta) 
= - \inner{\xi}{x}(x,x\lambda) 
\in W.
$$
This proves Step~1.

\medskip\noindent{\bf Step~2.}
{\it We prove the lemma.}

\medskip\noindent
Let $W\subset\H^2$ be a real linear subspace of
real dimension at most four.  Then the set
$\cM:=f^{-1}(W)=\left\{(x,\lambda)\in\Sp(1)\times\H\,|\,
(x,x\lambda)\in W\right\}$ is a smooth submanifold 
of $\Sp(1)\times\H$ of (real) dimension at most three
by Step~1. Hence the projection $\cM\to\H:(x,\lambda)\mapsto\lambda$
is not surjective by Sard's theorem.  Hence there exists 
an element $\lambda\in\H$ such that 
$\cM\cap(\Sp(1)\times\{\lambda\})=\emptyset$
and so $V_\lambda\cap W=0$. This proves Lemma~\ref{le:H}.
\end{proof}


\end{document}